\documentclass[letterpaper]{amsart}
\usepackage{amsmath}
\usepackage{amsthm}
\usepackage{amssymb}
\usepackage{amscd}          			%diagrams
\usepackage{color}  
\usepackage{bbm}            			%bold numbers
\usepackage{mathrsfs}           		%nice calligraphy
\usepackage[bookmarksnumbered,plainpages,backref]{hyperref} % Laurent i added new package

\usepackage[french,english]{babel}     		%if writing in French

\usepackage{graphicx}          		%uncomment for inclusion of graphic images

\usepackage{verbatim}           		%DO NOT comment this line

\usepackage{MnSymbol}

\theoremstyle{plain}
\newtheorem{Theorem}{Theorem}[section]
\newtheorem{Lemma}[Theorem]{Lemma}
\newtheorem{Corollary}[Theorem]{Corollary}
\newtheorem{Proposition}[Theorem]{Proposition}
\newtheorem{Conjecture}[Theorem]{Conjecture}
\newtheorem{theorem}[Theorem]{Theorem}
\newtheorem{proposition}[Theorem]{Proposition}
\newtheorem{corollary}[Theorem]{Corollary}
\newtheorem{lemma}[Theorem]{Lemma}

\theoremstyle{definition}
\newtheorem{Definition}[Theorem]{Definition}
\newtheorem{Example}[Theorem]{Example}
\newtheorem{definition}[Theorem]{Definition}

\theoremstyle{remark}
\newtheorem{Remark}[Theorem]{Remark}
\newtheorem{Claim}[Theorem]{Claim}

\newcommand{\bbC}{\mathbb{C}}       %complex numbers
\newcommand{\N}{\mathbb{N}}     %natural numbers
\newcommand{\R}{\mathbb{R}}     %real numbers
\newcommand{\fF}{\mathcal{F}}    

\newcommand{\calD}{\mathscr{D}}

\newcommand{\bW}{\mathbf{W}}

\DeclareMathOperator{\dist}{dist}			%distance
\DeclareMathOperator{\diver}{div}			%divergence
\DeclareMathOperator{\im}{im}				%image
\DeclareMathOperator{\supp}{supp}			%support
\usepackage[mathcal]{eucal}					%for the set of BV subsets of Rn

\newcommand{\bsA}{{\mathcal{A}}}
\newcommand{\bsB}{{\mathcal{B}}}

\newcommand{\bsL}{{\mathcal{L}}}

\newcommand{\bsT}{{\mathcal{T}}}

\newcommand{\LL}{\mathrm{L}}

\usepackage{ulem} % added by Tiago

\newcommand{\la}{\langle}
\newcommand{\ra}{\rangle}

\renewcommand{\geq}{\geqslant}
\renewcommand{\leq}{\leqslant}
\renewcommand{\ge}{\geqslant}
\renewcommand{\le}{\leqslant}
\renewcommand{\epsilon}{\varepsilon}

\newcommand{\liminfe}{\mathop{\underline{\lim}}}
\newcommand{\limsupe}{\mathop{\overline{\lim}}}

\usepackage{color}

\begin{document}

\title{On local continuous solvability of equations\\ associated to elliptic and canceling\\
linear differential operators}
\author{Laurent Moonens and Tiago Picon}
\date{\today}
\subjclass[2010]{Primary: 35J30;  Secondary: 35B45, 35F35, 46A03, 46A30, 46A32.}
\thanks{Laurent Moonens and Tiago Picon were partially supported by the Brazilian-French Network in Mathematics (RFBM) and the FAPESP grants  nos.~2017/17804-6, 2018/15484-7 and 2019/21179-5, respectively.}
%\begin{comment}

\maketitle
\begin{abstract}
Consider $A(x,D):C^{\infty}(\Omega,E) \rightarrow C^{\infty}(\Omega,F)$ {an elliptic and canceling}  linear differential operator  of order $\nu$ with smooth complex coefficients in $\Omega \subset \R^{N}$ from a finite dimension complex vector space $E$ to a finite dimension complex vector space $F$ and $A^{*}(x,D)$ {its} adjoint. In this work we characterize the (local) continuous solvability of the partial differential equation $A^{*}(x,D)v=f$ (in the distribution sense) for a given distribution $f$; more precisely we show that any $x_0\in\Omega$ is contained in a neighborhood $U\subset \Omega$ in which its continuous solvability is characterized by the following condition on $f$: for every $\epsilon>0$ and any compact set $K \subset \subset U$, there exists $\theta=\theta(K,\epsilon)>0$ such that the following holds for all smooth function $\varphi$ supported in $K$:%one has:%for every  i$\varphi \in \calD'{K}$,}
\begin{equation}\nonumber
\left| f(\varphi) \right| \leq \theta\|\varphi\|_{W^{\nu-1,1}}+\epsilon\|A(x,D) \varphi\|_{L^{1}},
\end{equation}
where $W^{\nu-1,1}$ stands for the homogenous Sobolev space of all $L^1$ functions whose derivatives of order $\nu-1$ belongs to $L^{1}(U)$.

This characterization implies and extends results obtained before for operators associated to elliptic complex of vector fields (see \cite{MP}); we also provide local analogues, for a large range of differential operators, to global results obtained for the classical divergence operator in \cite{BB1} and \cite{PP}. 
\end{abstract}
\bigskip

%Suppose that $\bsL=\left\{L_{1},\dots,L_{n}\right\}$ is a  system of linearly independent vector fields with smooth complex coefficients defined on an open set $\Omega \subset \R^{N}$. We may consider the operators $\nabla_{\bsL}\,u\doteq(L_{1}u,\dots,L_{n}u)$ for $u\in C^{\infty}(\Omega)$ and ${\rm div}_{\bsL}\,v\doteq\sum_{j=1}^{n}L_{j}v_{j}$ for $v\in C^{\infty}(\Omega,\Lambda^{1}\R^{n})$  which are precisely the operators $\nabla$ and $\rm{div}$ when $n=N$ and $L_{j}=\partial_{x_{j}}$. 

\section{Introduction}

Consider $\Omega \subseteq \R^{N}$ an open set and $A(\cdot,D)$ a linear differential operator of order $\nu$ with smooth complex coefficients in $\Omega$ denoted by:
% Let $E$ and 
%$F$ be finite dimensional complex vector spaces of dimension $n$ and  respectively and consider 
\[
A(x,D)=\sum_{|\alpha| \le\nu} 
a_\alpha(x)\partial^\alpha: C^{\infty}(\Omega;E)\rightarrow C^{\infty}(\Omega;F),
\]
where $E$ is a complex vector space of dimension $n$ and $F$ is a complex vector space of dimension $n'\ge n$.

A series of results concerning on local $L^{1}$ estimates for linear differential operators has been studied by J. Hounie and T. Picon in the setting of  elliptic systems of complex vector fields, complexes and pseudocomplexes (\cite{HP1}, \cite{HP2}). The following characterization of local $L^{1}$ estimates for operators $A(x,D)$ was proved in \cite{HP3}, namely:

\begin{Theorem}\label{thm.hp}Assume, as before, that $A(\cdot,D)$ is a linear differential operator of order $\nu$ between the spaces $E$ and $F$.  The following properties are equivalent:
\begin{enumerate}
  \item[(i)] $A(x,D)$ is elliptic 
and canceling (see below for a definition of those properties);
  \item[(ii)]\label{1.11}  every point $x_0\in\Omega$ is contained in a ball $B=B
(x_0,r)\subset\Omega$ such that the a priori estimate
\begin{equation}\label{hp3ineq}
\|u\|_{W^{\nu-1,N/(N-1)}}\le C \,\|A(x,D)u\|_{L^{1}},
\end{equation}
holds for some $C>0$ and all smooth functions $u \in C^{\infty}(B;E)$ having compact support in $B$.
\end{enumerate}
\end{Theorem}
Here, given $k\in\N$ and $1\leq p\leq\infty$, $W^{k,p}(\Omega)$ denotes the homogeneous Sobolev space of complex functions in $L^{p}(\Omega)$ whose weak derivatives of order $k$ belong to $L^{p}(\Omega)$, endowed with the (semi-)norm $\|u\|_{W^{k,p}}:=\sum_{|\alpha|=k}\| \partial_{\alpha}u \|_{L^{p}}$. %We will denote by $W^{-1,N/(N-1)}$ the dual of the space $W^{1,N}$. 

It turns out that elliptic linear differential operators 
that satisfy an \textit{a priori} estimate like \eqref{hp3ineq} can be characterized in terms of properties of 
their principal symbol $a_\nu (x,\xi)=\sum_{|\alpha|=\nu}a_\alpha(x)\xi^\alpha$. Recall 
that the \textit{ellipticity} of $A(x,D)$ at $x_0\in\Omega$ means that for every  $\xi\in \R^N
\setminus\{0\}$ the map $a_\nu(x_0,\xi):E\longrightarrow F$ is injective.

\begin{definition}\label{d2.1}
Let $x_0\in\Omega$. A linear partial differential operator $A(x,D)$ of order $\nu$ from $E$ to $F$ with principal symbol $a_\nu(x,\xi)$ 
that satisfies :
\[
\bigcap_{\xi\in\R^N\setminus\{0\}} a_\nu(x_0,\xi)[E]=\{0\} \tag{$\star$}
\]
is said to be \textit{canceling at $x_0$}. If $(\star)$ holds for every $x_0\in\Omega$ we say that $A(x,D)$ is \textit{canceling}.
\end{definition}
%{\Rd This definition is invariant under change of coordinates provided that we regard $ a_\nu(x,\xi)$ as a function defined on $T^*(\Omega)$ with values in $\mathcal{L}(E,F)$}.\todo{Maybe explain the notations and say which class of change of coordinates are allowed ?} 

Examples of canceling operators satisfying $(\star)$ can be founded in \cite{HP3}; this is the case in particular for operators associated to elliptic system of complex vector fields (see \cite{HP1}, \cite{HP2}). The canceling property for linear differential operators was originally defined by Van Schaftingen \cite{VS4} in the setup of homogeneous operators with constant coefficients $A(D)$ and  stands out by several applications (and characterizations) in the theory of \textit{a priori} estimates in $L^1$ norm (see for instance \cite{VS5}  for a brief description).  \\
%and extended for general linear differential operators (non-homogeneous) $A(x,D)$ with smooth complex coefficients by Hounie and Picon \cite{HP3}. The   .... \textcolor{red}{write more} 

%The following version of the $L^{1}$ Sobolev-Gagliardo-Nirenberg theorem associated to $\nabla_{\bsL}$ was proved in \cite{HP1}, namely:
%\begin{theorem}\label{thmHP}
%Assume that the system of vector fields $L_{1},...,L_{n}$, $n \geq 2$, is linearly independent  and elliptic. Then every point $x_{0} \in \Omega$ is contained in an open neighborhood $U \subset \Omega$ such that
%\begin{equation}\label{SGN}
%\|\varphi\|_{L^{N/N-1}}\leq C\sum_{j=1}^{n}\|L_{j}\varphi\|_{L^{1}}, \;\;\;\forall\;\varphi \in \calD (U),
%\end{equation}
%holds for $C=C(U)>0$. Conversely, if \eqref{SGN} holds then the system must be elliptic on $U$. 
%\end{theorem}
%\textcolor{blue}{I put this sentence in the Section 2: Here ellipticity means that for any \textit{real} $1$-form $\omega$  such that  $\langle\omega,L_j\rangle=0$ holds for $j=1,..,n$, we have $\omega=0$; equivalently, this also means that the second-order operator $\Delta_\bsL=\diver_{\bsL^*}(\nabla_\bsL \cdot)$ is elliptic.}

In this work, we are interested to study the (local) continuous solvability in the weak sense of the equation:
\begin{equation}\label{diverl}
A^{*}(x,D) v =f,
\end{equation}
where $A(x,D)$ is an elliptic and canceling linear differential operator. We use the notation  $A^{*}\;:=\;\overline{ A^{t}}$ where $\overline A$ denotes the operator obtained from $A$ by conjugating its coefficients and $A^t$ is its formal transpose~---~namely this means that, for all smooth functions $\varphi$ and $\psi$ having compact support in $\Omega$ and taking values in $E$ and $F$ respectively, we have:
 $$
 \int_\Omega A(x,D)\varphi \cdot \bar{\psi}=\int_\Omega \varphi \cdot \overline{A^*(x,D)\psi}.
 $$
 
Our main result is the following.
\begin{theorem}\label{mainthm}
Assume $A(x,D)$ is as before. Then every point $x_{0} \in \Omega$ admits an open neighborhood $U \subset \Omega$ such that for any $f \in \calD'({U})$, the equation (\ref{diverl}) is continuously solvable in $U$ if and only if $f$ is an $\bsA$-charge in $U$, meaning that for every $\epsilon>0$ and every compact set $K \subset \subset U$, there exists $\theta=\theta(K,\epsilon)>0$ such that one has:%for every  i$\varphi \in \calD'{K}$,
\begin{equation}\label{lstrong}
\left| f(\varphi) \right| \leq \theta\|\varphi\|_{W^{\nu-1,1}}+\epsilon\|A(x,D) \varphi\|_{L^{1}}, 
\end{equation}
for any smooth function $\varphi$ in $U$ vanishing outside $K$.
\end{theorem}

One simple argument (see Section \ref{sec2}) shows that the above continuity property on $f$ is a necessary condition for the continuous solvability of equation \eqref{diverl} in $U$. Theorem \ref{mainthm} asserts that the continuity property \eqref{lstrong} is also sufficient, under the canceling and ellipticity assumptions on the operator. The proof, which will be presented in Section \ref{sec3}, is based on a functional analytic argument inspired from \cite{PP} and already improved in \cite{MP} for divergence-type equations associated to complexes of vector fields (observe in particular that one recovers \cite[Theorem 1.2]{MP} when applying Theorem~\ref{mainthm} to the latter context). However, it should be mentioned here that by allowing in \eqref{diverl} a much larger class of (higher order) differential operators, that method of proof had to be very substantially refined, leading to the use of new tools. Applications of Theorem \ref{mainthm} are presented in the Section \ref{aplica}.

Assume for a moment that $A(x,D)$ be elliptic. We point out the canceling assumption - characterized by inequality \eqref{hp3ineq} -  plays a fundamental role in our argument in the proof of Theorem \ref{mainthm}. However, we should emphasize that this property might not be necessary to obtain a characterization of continuous solutions to the equation \eqref{diverl} formulated along the previous lines. In the context of the Poisson equation with measure data $\Delta u=\mu$ (where the Laplacean operator is \textit{not} a canceling operator), it follows indeed from a result by Aizenman and Simon \cite[Theorem~4.14]{AS} (see also Ponce \cite[Proposition~18.1]{PONCE}, where this question is studied in a luminous fashion) that, given a smooth, bounded open set $\Omega$ in $\R^n$ and a measure $\mu$ in $\Omega$, the Dirichlet problem associated with $\Delta u=\mu$ in $\Omega$ has a continuous solution in $\bar{\Omega}$ if and only if, for every $\epsilon>0$, there exists $\theta>0$ such that for any $\varphi\in C^\infty_0(\bar{\Omega})$, one has:
$$
\left|\int_\Omega \varphi\,d\mu\right|\leq \theta\|\varphi\|_{L^1}+\epsilon\|\Delta \varphi\|_{L^1}.
$$
This result, however, is proved using very different techniques than the ones developed here.
  
%The organization of the paper is as follows

\section{Preliminaries and notations}
We always denote by $\Omega$ an open set of $\R^{N}$, $N\geq 2$. Unless otherwise specified, all functions are complex valued and the notation $\int_A f$ stands for the Lebesgue integral $\int_A f(x)dx$.  As usual, $\calD(\Omega)$ and $\calD'(\Omega)$ are the spaces of complex test functions and distributions, respectively. When $K\subset \subset \Omega$ is a compact subset of $\Omega$, we let $\calD_{K}(\Omega):=\calD(\Omega)\cap\mathcal{E}'(K) $, where $\mathcal{E}'(K)$ is the space of all distributions with compact support in $K$. Since the ambient field is $\bbC$, we identify (formally) each $f\in L^1_{loc}(\Omega)$  with the distribution $T_f\in\calD'(\Omega)$ given by $T_f(\varphi)=\int_\Omega f\bar{\varphi}$. We also consider $C(\Omega)$ the space of all continuous functions in $\Omega$. %We also introduce the notation $\|\nabla_{\bsL}\varphi\|_{p}:=\sum_{j=1}^{n}\|L_{j}\varphi\|_{p}$ (where $\|\cdot\|_{p}$ is the standard norm in $L^{p}(\Omega)$) for $1\leq p\leq \infty$.
When working with objects in a function space taking values in a finite-dimensional (normed) vector space $E$, we shall indicate it as a second argument (\textit{e.g.} $C(\Omega,E)$ will denote the space of all $E$-valued continuous vector fields $v:\Omega\to E$).
Finally we use the notation $f\lesssim g$ to indicate the existence of an universal constant $C>0$, independent of all variables and unmentioned parameters, such that one has $f \leq Cg$.

\subsection*{{Some Sobolev spaces}}
Given a finite-dimensional (complex, normed) vector space $E$, we denote as before by $W^{k,p}(\Omega; {E})$ for $k \in \N$ and $1\leq p \leq \infty$, the homogeneous Sobolev space of functions in $L^{p}(\Omega; {E})$ whose weak derivatives of order $k$  belong to $L^{p}(\Omega; {E})$, endowed with the (semi-)norm $\|u\|_{W^{k,p}}=\sum_{|\alpha|=k}\| \partial_{\alpha}u \|_{L^{p}}$; {we also denote by $W^{k,p}_c(\Omega; { E})$ the space of its elements having compact support in $\Omega$}. Given $1< p\leq \infty$ and $k\in\N^*$ we also define the space $W^{-k,p}_{\mathrm{loc}}(\Omega; {E^*}):=(W_c^{k,p'}(\Omega; {E}))^*$ as the space of distributions $f\in\calD'(\Omega; E^*)$  enjoying the following property: for all $K\subset\subset \Omega$, there exists $C_K>0$ such that for all $\varphi\in\calD_K(\Omega; {E})$, one has:
$$
|\la f,\varphi\ra|\leq C_K \|\varphi\|
_{W^{k,p'}(\Omega)}=C_K\sum_{|\alpha|=k} \|\partial^\alpha\varphi\|_{p'},
$$
where $1\leq p'< \infty$ is defined by $\frac 1p +\frac{1}{p'}=1$.

Given $k\in\N$ and $1\leq p\leq\infty$, one also defines the (classical, inhomogeneous) Sobolev space $\bW^{k,p}(\Omega; {E})$ as the space of complex functions in $L^{p}(\Omega; {E})$ whose weak derivatives up to order $k$  belong to $L^{p}(\Omega; {E})$, endowed with the norm $\|u\|_{\bW^{k,p}}:=\sum_{|\alpha|\leq k}\| \partial_{\alpha}u \|_{L^{p}}$. We denote finally by $\bW^{k,p}_0(U; {E})$ the completion of the space $\calD(\Omega; {E})$ in $\bW^{k,p}(\Omega; {E})$. The space $\bW^{k,p}_{0}(\Omega; {E})$ is classically reflexive and separable for $1<p<\infty$ (see \textit{e.g.} \cite[p.~64]{AF}).

\begin{Lemma}\label{lemma.hb}
Fix $k\in\N^*$ and $1<p\leq \infty$.
Given $f\in W^{-k,p}_{\mathrm{loc}}(\Omega;{ E^*})$ and $K\subset\subset\Omega$ there exists $(g_\alpha)_{|\alpha|= k}\subseteq L^p(\Omega; {E^*})$ for which one has, for all $\varphi\in\calD_K(\Omega; {E})$:
$$
\la f,\varphi\ra=\sum_{|\alpha|= k} \int_\Omega g_\alpha \overline{\partial^\alpha\varphi}.
$$
\end{Lemma}
{
\begin{Remark}
In the latter expression, and throughout this paper, one uses, given $e\in E$ and $e^*\in E^*$ (and in any finite-dimensional duality setting), the notation $e^*e$ instead of $\la e^*,e\ra$.
\end{Remark}}
\begin{proof}
Fix $f\in W^{-k,p}_{\mathrm{loc}}(\Omega; {E^*})$ and $K\subset\subset \Omega$.
Denote by $M$ the number of multi-indices $\alpha\in\N^n$ with $|\alpha|=k$ and let:
$$
X:=\{(u_\alpha)_{|\alpha|=k}: \text{there exists } \varphi\in\calD_K(\Omega; {E})\text{ with } u_\alpha=\partial^\alpha\varphi\text{ for all }|\alpha|=k\},
$$
be endowed with the norm $\|(u_\alpha)_{|\alpha|=k}\|:=\sum_{|\alpha|=k} \|u_\alpha\|_{p'}$~---~we hence see it as a subspace of $L^{p'}(\Omega,{E^M})$.

Now define a linear functional $F$ on $X$ in the following way: if $(u_\alpha)_{|\alpha|=k}\in X$ is given, there exists a unique $\varphi\in\calD_K(\Omega;{E})$ with $u_\alpha=\partial^\alpha \varphi$ for all $|\alpha|=k$; we then let $\la F, (u_\alpha)\ra:=\la f,\varphi\ra$.
There holds:
$$
|\la F,(u_\alpha)\ra|\leq C_K\sum_{|\alpha|=k} \|u_\alpha\|_{p'},
$$
for all $(u_\alpha)_{|\alpha|=k}\in X$.
By the Hahn-Banach theorem, $F$ extends to a continuous linear functional on $L^{p'}(\Omega,E^M)$ satisfying
$$
|\la F,(u_\alpha)\ra|\leq C_K\sum_{|\alpha|=k} \|u_\alpha\|_{p'},
$$
for all $(u_\alpha)\in L^{p'}(\Omega,{E^M})$; there hence exists $(g_\alpha)\in L^p(\Omega;{(E^*)^M}) = (L^{p'}(\Omega;E^M))^*$ such that one has:
$$
\la F, (u_\alpha)\ra=\sum_{|\alpha|=k} \int_\Omega g_\alpha \bar{u}_\alpha,
$$
for all $(u_\alpha)\in L^{p'}(\Omega,{E^M})$. We hence get in particular, for $\varphi\in\calD_K(\Omega;{ E})$:
$$
\la f,\varphi\ra=\la F, (\partial^\alpha\varphi)_{|\alpha|=k}\ra =\sum_{|\alpha|=k} \int_\Omega g_\alpha\overline{\partial^\alpha\varphi}.
$$
The proof is complete.
\end{proof}

\begin{Example}\label{ex.sob-neg}
Assume $\Omega$ to be an open set and fix $k\in\N$. Then one has $L^N_{\mathrm{loc}}(\Omega;{E^*})\subseteq W^{-k,N}_{\mathrm{loc}}(\Omega;{E^*})$.

To see this, observe the statement is obvious in case $k=0$. Hence assume $k\in\N^*$ and fix $f\in L^N_{\mathrm{loc}}(\Omega;{E^*})$. Let $K\subset\subset\Omega$ be compact and compute now for $\varphi\in\calD_K(\Omega;{E})$:
$$
\left|\int_\Omega f\bar{\varphi}\right|\leq \|f\|_{L^N(K;{E^*})}\|\varphi\|_{L^{N/(N-1)}(K;{E})} =  \|f\|_{L^N(K;{E^*})}\|\varphi\|_{L^{N/(N-1)}(\R^N;{E})}.
$$
Yet we get, according to the Sobolev-Gagliardo-Nirenberg inequality (which we shall refer to as the ``SGN'' inequality in the sequel):
$$
\|\varphi\|_{L^{N/(N-1)}(\R^N;{E})}\leq \kappa(N)\sum_{|\alpha|=1} \|\partial^\alpha\varphi\|_{L^1(\R^N;{E})} =  \kappa(N)\sum_{|\alpha|=1} \|\partial^\alpha\varphi\|_{L^1(\Omega;E)},
$$
and hence we find, for all $\varphi\in\calD_K(\Omega)$:
$$
\left|\int_\Omega f\bar{\varphi}\right|\leq \kappa(N) \|f\|_{L^N(K;{E^*})} \|\varphi\|_{W^{1,1}(\Omega;{E})}.
$$
If $k=2$, then for all $|\alpha|=1$ we have, using H\"older's inequality and SGN inequality again:
\begin{multline*}
 \|\partial^\alpha\varphi\|_{L^1(\Omega;{ E})}= \|\partial^\alpha\varphi\|_{L^1(K;{E})}\leq |K|^{\frac 1N} \|\partial^\alpha\varphi\|_{L^{N/(N-1)}(\R^N;{E})}\\\leq |K|^{\frac 1N} \sum_{|\beta|=1} \|\partial^{\alpha+\beta}\varphi\|_{L^1(\R^N;{E})}=|K|^{\frac 1N} \sum_{|\beta|=1} \|\partial^{\alpha+\beta}\varphi\|_{L^1(\Omega;{E})},
 \end{multline*}
 which implies, for all $\varphi\in\calD_K(\Omega)$:
 $$
 \left|\int_\Omega f\bar{\varphi}\right|\leq \kappa(N)\|f\|_{L^N(K;{E^*})}|K|^{\frac 1N} \sum_{|\alpha|=2} \|\partial^\alpha \varphi\|_{L^1(\Omega;{E})}=
  \kappa(N)\|f\|_{L^N(K;{E^*})}|K|^{\frac 1N}\|\varphi\|_{W^{2,1}(\Omega;{E})}.
  $$
 One proves inductively for a general $k\in\N^*$ one has:
   $$
 \left|\int_\Omega f\bar{\varphi}\right|\leq \kappa(N)\|f\|_{L^N(K;{E^*})}|K|^{\frac{k-1}{N}}\|\varphi\|_{W^{k,1}(\Omega;{E})}
 \leq \kappa(N)\|f\|_{L^N(K;{E^*})} |K|^{\frac kN } \|\varphi\|_{W^{k,N/(N-1)}(\Omega;{E})},
  $$
where H\"older's inequality is used again; this means finally that one has $f\in W^{-k,N}_{\mathrm{loc}}(\Omega;{E^*})$.
\end{Example}

\begin{Remark}\label{rmk.sobolev}
Using approximation by smooth functions and a recursive use of the SGN  inequality as done in the above example, one shows that, given $K\subset\subset\Omega$ a compact set and an integer $k\in\N$, there exists a constant $C(K,k)>0$ such that for any $g\in W^{k, N/(N-1)}(\Omega;{E})$ satisfying $\supp g\subseteq K$, one has $g\in \bW^{k,N/(N-1)}_{0}(\Omega;{E})$ and:
$$
\|g\|_{\bW^{k,N/(N-1)}(\Omega;{E})}\leq C(K,k) \|g\|_{W^{k,N/(N-1)}(\Omega;{E})}.
$$
\end{Remark}

%We now study some properties of elliptic systems of vector fields.

%
%
\section{Canceling and elliptic differential operators}\label{sec0}

Given $A(x,D)$ as before, the $2\nu$-order differential operator $\Delta_{\bsA}:=A^{\ast}(\cdot,D)\circ A(\cdot,D)$ may be regarded as an elliptic pseudodifferential operator with symbol in the H\"ormander class $S^{2\nu}(\Omega):=S^{2\nu}_{1,0}(\Omega)$\footnote{{Since we will only work with symbols of type (1,0), the type
will be omitted in the notation; concerning pseudo-differential operators we
refer, for instance, to \cite[Chapter 3]{Hor} and \cite{Taylor}.}}, so that there exist properly supported pseudodifferential operators $q, \tilde{q} \in OpS^{-2\nu}(\Omega)$ {(parametrixes)} and $r, \tilde{r} \in OpS^{-\infty}(\Omega)$ for which one hasl, for any $f \in C^{\infty}(\Omega,F)$:
\begin{equation}\label{pseudo}
\Delta_{\bsA} q(x,D)f+r(x,D)f=\tilde{q}(x,D)\Delta_{\bsA}f+\tilde{r}(x,D)f=f.
\end{equation}
Writing $\Delta_{\bsA} q(x,D)f={A^{*}}(x,D)u$ for $u=A(x,D)q(x,D)f$ we then get: 
\begin{equation}\nonumber
{A^{*}}(x,D)u-f=r(x,D)f
\end{equation}
for every  $f \in C^{\infty}(\Omega,F)$.

%As application from the previous identity we present the following a priori estimates
{
\begin{proposition}\label{prop.gradmod}
Assume that $A(x,D)$ is as before. Then for every point $x_{0} \in \Omega$ and any $0<\beta<1$, there exist an open ball $B=B(x_{0},\ell) \subset \Omega$ and a constant $C=C(B)>0$ such that, for all $\varphi\in\calD(B,E)$, one has:
\begin{equation}\label{jh}
\sum_{|\alpha|=\nu-1} \|\partial^\alpha \varphi \|_{{1-\beta,1}}\leq C\| A(x,D)\varphi\|_{L^{1}}.
\end{equation}
\end{proposition}

The previous inequality states the embedding into $L^1$ of some version of a fractional Sobolev space $W^{1-\beta,1}_{c}(B)$ that can be defined according to the following procedure. Given $B=B(x_0,\ell)$ a ball consider $\tilde{B}=B(x_0,2\ell)$ the ball with the same center as $B$ but twice its radius. Let $\psi \in \calD(\tilde{B})$ satisfy $\psi(x) \equiv 1$ on $B$ and define  $\Lambda_{\gamma} := \Lambda_{\gamma}(x,D)$ the pseudodifferential operator with symbol $\lambda_{\gamma}(x,\xi)=\psi(x)(1+4\pi^{2}|\xi|^{2})^{\gamma} \in S^{\gamma}(\R^N)$. 
Denote then by $W^{\gamma,p}_{c}(B)$ the set of distributions with compact support $f \in \mathcal{E}'(B)$ such that one has $\Lambda_{{\gamma}}f \in \LL^{p}(\R^{N})$; one endows it with the semi-norm $\|f\|_{{\gamma,p}}\;:=\;\|\Lambda_{{\gamma}}u\|_{p}$. Note that the space  $W^{\gamma,p}_{c}(B)$ is independent of the choice of $\psi$, \textit{i.e.} that if $\psi_{2}(x), \psi_{1}(x) \in \calD(\widetilde{B})$ satisfy $\psi_{1}(x)=\psi_{2}(x) \equiv 1$ on $B$, then $\|\Lambda_{\gamma, \psi_{1}}f\|_{L^{p}}\; \cong \;\ \|\Lambda_{\gamma, \psi_{2}}f\|_{L^{p}}$.

%Let $\lambda>0$ and $1\leq p <\infty$. We define for $g \in \varphi\in\calD(U)$ the functional $\|g\|_{\lambda,p}:= \|J_{\alpha} g\|_{L^p}$ where

%\textcolor{red}{In view of the content of the Appendix, shouldn't we already here use $\Lambda_\alpha$ instead of $J_\alpha$ to be consistent ?}

%In the above statement, the operator $J_{\lambda}:=J_{\lambda}(D)$ for $\lambda>0$ is the pseudodifferential operator   defined by 
%$$J_{\lambda}f(x)=\int_{\R^{N}}e^{2\pi i x \cdot \xi}b_{\lambda}(\xi)\hat{f}(\xi)d\xi, \quad f \in S(\mathbb{R}^{N}),$$
%where the symbol $b_{\lambda}(x,\xi)=\langle \xi \rangle^{\lambda}:=(1+4\pi^{2}|\xi|^{2})^{-\lambda/2}$, independent of $x$, belongs to the H\"ormander class $S^{-\lambda}_{1,0}(\R^{N})$.
%The operator $J_{\lambda}$, usually denoted by $(1-\Delta)^{-\lambda/2}$, allows us to introduce a type of nonhomogeneous fractional Sobolev also known potential space (see \cite[p.135]{S2}) denoted by 
%$\mathcal{L}_{\lambda}^{p}(\R^{N})$ for $1\leq p <\infty$, defined
%as the set of functions $f \in L^{p}(\R^{N})$ which can be written as $g=J_{\lambda}f$ for some $ g \in L^{p}(\R^{N})$, endowed with the norm $\|f\|_{{\lambda,p}}:=\|g\|_{L^{p}}$.   %As a consequence of the continuity property of the action of the Bessel potential on Lebesgue spaces (see for instance \cite[Theorem~2.5]{AH}), the inclusion $W^{\alpha,p}(\R^{N}) \subset L^{p}(\R^{N})$ is continuous for all $1\leq p<\infty$.
\begin{proof}
Fix $\alpha$ a multi-index with $|\alpha|=\nu-1$.
Let $h=A(x,D)\varphi$. Thanks to identity \eqref{pseudo} and to the calculus of pseudodifferential operators we have 
$$\Lambda_{1-\beta}\partial^\alpha\varphi=p(x,D)h+r'(x,D)\varphi,$$
where  $p(x,D):=\Lambda_{1-\beta} \partial^\alpha q(x,D) {A^{\ast}}(x,D) \in OpS^{-\beta}$ %, $\tilde{p}(x,D):=J_{1-\beta}(D)q(x,D) {A^{\ast}}(x,D)[A(x,D),\partial^\alpha] \in OpS^{\nu-\beta}$} 
and $r'(x,D):=\Lambda_{1-\beta} \partial^\alpha r(x,D) \in OpS^{-\infty}$. As a consequence of \cite[Theorem 6.1]{MP}  we have $\|p(x,D)h\|_{L^{1}}\lesssim \| A(x,D)\varphi\|_{L^{1}}$, which implies:
 \begin{equation}\nonumber
\| \Lambda_{\beta-1}\partial^\alpha\varphi\|_{L^{1}}\leq C\| A(x,D)\varphi\|_{L^{1}} +\|r'(x,D)\varphi\|_{L^{1}}.
\end{equation}
As the second term on the right side may be absorbed (see \cite[p.~798]{HP1}), shrinking the neighborhood if necessary, we obtain the estimate \eqref{jh} after recalling estimate \eqref{1.11} in Theorem~\ref{thm.hp} above.
\end{proof}}
%\todo{Shouldn't we work, in the latter proof, with $\Lambda_{1-\beta}$ as in the paragraph before the statement?}
%{The boundedness in $\LL^1$ norm of the pseudodifferential operators with negative order {follow from the}  integrability property of the kernel due itself to {a} pointwise control obtained in \cite{AH}.  Another fundamental tool from pseudodifferential operators theory, inspired in the recent results obtained in \cite{HKP}, asserts that the embedding  $W_{c}^{1-\beta, 1}(B):=W^{1-\beta,1}(\R^{N})\cap \mathcal{E}'(B)$, where $B$ is a generic ball, {into} $L^{1}(\R^{N})$ is compact. These results are stated in the Appendix and will be proved there for sake of completeness.}

\section{Functions of bounded variation associated to $A(x,D)$}\label{sec1}

\subsection{Basic definitions; approximation and compactness}
Let $W^{k,p}_{c}(\Omega;E)$ be the linear space of all complex functions in $W^{k,p}(\Omega;E)$  whose support is a compact subset of $\Omega$.

The following definition of variation associated to $A(x,D)$ of $g\in W^{\nu-1,1}_c(\Omega; E)$ recalls the classical definition of variation when $\nu=1$ and $A(x,D)=\nabla$. It has been formulated for real vector fields by N. Garofalo and D. Nhieu \cite{GAROFALONHIEU} and adapted for complex vector fields in \cite{MP}.
\begin{Definition}
Given $g\in W^{\nu-1,1}_c(\Omega;E)$ \text and $U\subseteq\Omega$ an open set, one calls the extended real number:
$$
\|D_{\bsA} g\|(U):=\sup\left\{\left|\int_\Omega g\,\overline{{A^*(\cdot,D)} v}\right|: v\in C^\infty_c(\Omega; F^*), \supp v\subseteq U, \|v\|_\infty\leq 1\right\},
$$
the \textit{(total) $\bsA$-variation associated to $A(x,D)$ of $g$ in $U$} and we let $\|D_\bsA g\|:=\|D_\bsA g\|(\Omega)$ in case there is no ambiguity on the open set $\Omega$.
We denote by $BV_{\mathcal{A},c}(\Omega)$ the set of all $g\in W^{\nu-1,1}_c(\Omega;E)$ with $\|D_{\bsA} g\|<+\infty$.

Given $g\in BV_{\mathcal{A},c}(\Omega)$, we denote by $D_\mathcal{A} g$ the unique $F$-valued Radon measure satisfying:
\begin{equation}\label{eq.def-mes}
\int_\Omega g\,\overline{A^*(\cdot,D) v}=\int_\Omega \bar{v}\cdot d [D_\bsA g],
\end{equation}
for all $v\in C^\infty_c(\Omega,F^*)$. It is clear by definition that $\|D_\bsA g\|$ is also the total variation in $\Omega$ of $D_\bsA g$.
\end{Definition}

\begin{Remark}Given $g\in BV_{\bsA,c}(\Omega)$, one has {$\supp D_\bsA g\subseteq \supp g$}. Indeed, given $x\in \Omega\setminus\supp g$, find a radius $r>0$ for which one has $B(x,r)\subseteq \Omega\setminus \supp g$. It is clear according to (\ref{eq.def-mes}) that for any $v\in C^\infty_c(B(x,r),F^*)$ we then have $D_\bsA g(v)=0$. Hence we also get $D_\bsA g(v)=0$ for all $v\in C_c(B(x,r),F^*)$, which ensures that one has $x\notin\supp D_\bsA g$ and finishes to show the inclusion $\supp D_\bsA g\subseteq\supp g$.
\end{Remark}
\begin{Remark}\label{rmk.sci}
It follows readily from the previous definition that  if $(g_i)\subseteq BV_{\bsA,c}(\Omega)$ converges in $\LL^1$ to $g\in W^{\nu-1,1}_c(\Omega)$, one then has $g\in BV_{\bsA,c}(\Omega)$ and:
$$
\|D_\bsA g\|\leq\liminfe_{i}\|D_\bsA g_i\|.
$$
We shall refer to this in the sequel as the \textit{lower semi-continuity} of the $\bsA$-variation.
\end{Remark}

We say that a sequence $(f_{i})_{i}$ of functions with complex values defined on open set $\Omega \subset \R^{N}$ is \textit{compactly supported in $\Omega$} if there is a compact set $K \subset\subset \Omega$ such that one has $\supp f_i \subseteq K$ for every $i$.

We shall make an extensive use of the following concept of convergence.
\begin{Definition}\label{def.dc}
Given $g \in W^{\nu-1,1}_{c}(\Omega;E)$ and a sequence $(\varphi_i)_{i}\subseteq\calD(\Omega; E)$ we shall write $\varphi_i \twoheadrightarrow g $ in case the following conditions hold:
\begin{enumerate}
\item[(i)] $(\varphi_{i})$ converges to $g$ in $W^{\nu-1,1}$ norm;
\item[(ii)] $(\varphi_i)$ is compactly supported in $\Omega$;
\item[(iii)] $\sup_i\| A \varphi_{i}\|_1<+\infty$.
\end{enumerate}
\end{Definition}

The following lemma is a Friedrich's type lemma; in the case where $A$ is a system of real vector fields, it reminds a result by N. Garofalo and D. Nhieu \cite[Lemma~A.3]{GAROFALONHIEU}.
In order to state it, fix $\eta\in\calD(\R^n)$ a radial function with nonnegative values, satisfying $\supp \eta\subseteq B[0,1]$ and $\int_{\R^n} \eta=1$, and, for each $\epsilon>0$, define $\eta_\epsilon\in\calD(\R^n)$ by $\eta_\epsilon(x):=\epsilon^{-N}\eta(x/\epsilon)$.
\begin{Lemma}\label{lemma4.5}
\ For any $g\in BV_{\bsA,c}(\Omega)$, one has:
$$
\lim_{\epsilon\to 0} \| A(\eta_\epsilon * g)-\eta_\epsilon*(D_{\bsA} g)\|_{L^1(\Omega)}=0.
$$
\end{Lemma}
\begin{proof}
We can assume for simplicity $E=\R^{n}$ and $F=\R^{n'}$.
Let us first assume that $\nu=1$ and write $A=c+\sum_{j=1}^m a_j\partial_j$ where $c$ and $a_j$, $1\leq j\leq m$ are locally Lipschitz functions. Write now, for $x\in\Omega$ and $\epsilon>0$ small enough so that one has $|c(y)-c(z)|\leq \delta$ for all $y,z\in\Omega$ with $|y-z|\leq\epsilon$:
$$
|c(x)(\eta_\epsilon*g)(x)-[\eta_\epsilon* (cg)](x)|\leq \int_{B(x,\epsilon)} |c(x)-c(y)||g(y)|\eta_\epsilon(x-y)\,dy\leq\delta \int_{B(x,\epsilon)} |g(y)|\eta_\epsilon(x-y)\,dy.
$$
We hence have, for $\epsilon>0$ small enough, using Fubini's theorem:
$$
\|c(\eta_\epsilon*g)-\eta_\epsilon*(cg)\|_{L^1(\Omega)} \leq \delta \|g\|_{L^1(\Omega)}.
$$
Writing now $L=\sum_{j=1}^N a_j\partial_j$, it now follows from \cite[Lemma~A.3]{GAROFALONHIEU} that one also has:
$$
\lim_{\epsilon\to 0}\| L(\eta_\epsilon*g)-\eta_\epsilon*(Lg)\|_{L^1(\Omega)}=0.
$$
This finishes the proof in case $\nu=1$.

In case $\nu>1$, it suffices to consider the case $A=a \partial^\alpha$ for some multi-index $\alpha\in\N^n$ with $|\alpha|\leq\nu$, and some locally Lipschitz function $a$. The case $|\alpha|\leq 1$ being already dealt with using the preceding part of the proof, we can assume $A=A'\partial_j$ where $A'=a'\partial^{\alpha'}$ for $|\alpha'|\geq 1$. Since we can now write (using the notation $[\cdot,\cdot]$ for the commutator):
$$
[A,\eta_\epsilon * (\cdot)]g=[A',\eta_\epsilon*(\cdot)] \partial_i g,
$$
it follows inductively from the fact that $\partial_i g\in W_{c}^{\nu-2,1}(\Omega)\subseteq L_{c}^1(\Omega)$ and from the preceding part of the proof, that one has:
$$
\|[A,\eta_\epsilon * (\cdot)]g\|_{L^1(\Omega)}\to 0, \quad\epsilon\to 0.
$$
The proof is complete.
\end{proof}

We now obtain an analogous result, in $BV_{\bsA,c}$, to the standard approximation theorem for $BV_c$ functions. 
\begin{lemma}\label{lemmaaprox}

For any $g \in BV_{\bsA,c}(U)$, there exists a sequence $\left\{ \varphi_{i} \right\}_{i} \subset \calD(U) $ such that one has $\varphi_i\twoheadrightarrow g$ and, moreover:
$$
\|D_\bsA g\|=\lim_i\| A(\cdot,D) \varphi_{i}\|_1.
$$
\end{lemma}
\begin{proof}

Fix now $g\in BV_{\bsA,c}(\Omega)$ and define for $0<\epsilon<\dist(\supp g, \complement\Omega)$ a function $g_\epsilon\in\calD(\Omega)$ by the formula:
$$
g_\epsilon:=\eta_\epsilon* g.
$$
It is easy to see that one has $g_\epsilon\to g$ in $L^1(\Omega)$ and that there exists a compact set $K\subset\subset\Omega$ such that one has $\supp g_\epsilon\subseteq K$ for all $\epsilon>0$ small enough.

On the other hand, observe that according to the previous lemma, one can write for $\epsilon>0$ small enough:
\begin{equation}\label{eq.friedrichs}
A(\eta_\epsilon*g)=\eta_\epsilon *(D_\bsA g)+H_{\epsilon}(g),
\end{equation}
where also $\|H_{\epsilon} (g)\|_1\to 0$, $\epsilon\to 0$.

Fix now $v\in C^\infty_c(\Omega, F^*)$ a smooth vector field satisfying $\|v\|_\infty\leq 1$ and compute for $\epsilon>0$ small enough:
{
$$
\la \eta_\epsilon *(D_\bsA g),\bar{v}\ra:=\la D_{\bsA} g, \eta_\epsilon*\bar{v}\ra=\int_\Omega \overline{\eta_\epsilon*v}\cdot d[D_\bsA g]=\int_\Omega g \overline{A^*(\cdot, D)(\eta_\epsilon*v)}.
$$
so that one also has:
}
$$
\left|\int_\Omega A (\eta_\epsilon*g)\cdot \bar{v}\right|\leq \left|\int_\Omega g\,\overline{A^* (\eta_\epsilon*v)}\right|+\|H_\epsilon(g)\|_1\leq \|D_{\bsA} g\|+\|H_\epsilon(g)\|_1.
$$

We hence get, by duality:
$$
\|A(\eta_\epsilon* g)\|_1\leq\|D_\bsA g\|+\|H_\epsilon(g)\|_1,
$$
and the result follows from the aforementioned property of $H_\epsilon(g)$ when $\epsilon$ approaches $0$, and from the lower semicontinuity property already mentioned.
\end{proof}

The following proposition is a compactness result in $BV_{\bsA}$.
\begin{Proposition}\label{prop.compacite}
Assume that the open set $U\subseteq \Omega$ supports a SGN inequality of the type appearing in [Theorem~\ref{thm.hp}, {\eqref{hp3ineq}}] as well as an inequality of type \eqref{jh} for some $0<\beta<1$. If $(g_i)\subseteq BV_{\bsA,c}(U)$ is compactly supported in $U$ and if moreover one has:
$$
\sup_{i} \|D_\bsA g_i\|<+\infty,
$$
then there exists $g\in BV_{\bsA,c}(U)$ and a subsequence $(g_{i_k})\subseteq (g_i)$ converging to $g$ in $W^{\nu-1,1}$.
\end{Proposition}
\begin{proof}
Choose a compact set $K\subset\subset U$ for which one has $\supp g_i\subseteq K$ for all $i$. Choose also, according to Lemma~\ref{lemmaaprox}, a sequence $(\varphi_i)\subseteq\calD(U)$ and a compact set $K'\subset\subset U$ satisfying the following conditions for all $i$:
$$
\supp\varphi_i\subseteq K',\quad \|g_i-\varphi_i\|_{W^{\nu-1,1}}\leq 2^{-i}\quad\text{and}\quad \|A\varphi_i\|_1\leq \|D_\bsA g_i\|+1.
$$
We hence have $\sup_i \| A\varphi_i\|_1<+\infty$ while it is clear that $(\varphi_i)$ is compactly supported and satisfies $\|g_i-\varphi_i\|_{W^{\nu-1,1}}\to 0$, $i\to\infty$.

Now fix $0<\beta<1$, a multi-index $\alpha\in\N^N$ satisfying $|\alpha|=\nu-1$ and observe that the sequence $(\psi_i)_{i}$ also satisfies, according to (\ref{jh}):
$$
\sup_i \|\partial^\alpha\varphi_i\|_{1-\beta,1}=\sup_i \|\Lambda_{1-\beta} \partial^\alpha\varphi_i\|_1\leq C\sup_i \|A \varphi_i\|_1<+\infty.
$$
It hence follows from the compactness of the inclusion of $W^{1-\beta,1}_c(U)\subset \subset\LL^1(U)$ (see \cite[Theorem 6.2]{MP}) that there exists $h^\alpha\in \LL^1_c(U)$ and a subsequence $(\varphi_{i_k^\alpha})\subseteq (\varphi_i)$ such that $\partial^\alpha \varphi_{i_k^\alpha}$ converges to $h^\alpha$ in $\LL^1(U)$. This yields a subsequence $(\varphi_{i_k})\subseteq (\varphi_i)$ such that, for all $\alpha\in\N^N$ with $|\alpha|=\nu-1$, one has $\partial^\alpha \varphi_{i_k}\to h^\alpha$, $k\to\infty$. Using the Rellich-Kondrachov theorem (see \textit{e.g.} Ziemer \cite[Theorem~2.5.1]{ZIEMER}), we can moreover assume that $\varphi_{i_k}\to g$ in $\LL^1_c(U)$. According to the closing lemma (see Willem \cite[Lemma~6.1.5]{WILLEM}), we then have $h^\alpha=\partial^\alpha g$ for all $\alpha\in\N^N$ with $|\alpha|=\nu-1$. This ensures  $g\in W^{\nu-1,1}_c(U)$ and the convergence of $(g_{i_k})$ to $g$ in $W^{\nu-1,1}(U)$. Moreover, the semicontinuity property of the $\bsA$-variation yields $g\in BV_{\bsA,c}(U)$, which terminates the proof.
\end{proof}
\begin{Remark}
According to Theorem~\ref{thm.hp} and Proposition~\ref{prop.gradmod}, we see that if one assumes $A$ to be elliptic and canceling, each point $x_0\in\Omega$ is contained a neighborhood $U\subseteq\Omega$ satisfying the hypotheses of the previous proposition. \end{Remark}

\subsection{A Sobolev-Gagliardo-Nirenberg inequality in $BV_\bsA$} As announced we get the following result:

\begin{proposition}\label{prop.SGNBV}
Let $A(x,D)$ be as before. Then every point $x_{0} \in \Omega$ is contained in an open neighborhood $U \subset \Omega$ such that the inequality:
\begin{equation}\label{dl}
\|g\|_{{W^{\nu-1,N/N-1}}}\leq C\|D_{\bsA}g\|,
\end{equation}
holds for all $g\in BV_{\bsA,c}(U)$, where $C=C(U)>0$ is a constant depending only on $U$.
\end{proposition}
\begin{proof}
Fix $x_{0} \in \Omega$. It follows from Theorem \ref{thm.hp} that there exists a neighborhood $U \subset \Omega$ of $x_{0}$ and $C=C(U)>0$ such that, for all $\varphi\in\calD(U;E)$, one has:
\begin{equation}\nonumber
\|\varphi\|_{W^{\nu-1,{N/N-1}}}\leq C \| A(\cdot,D) \varphi\|_{{1}}.
\end{equation}
Then given $g \in BV_{\bsA,c}(U)$ consider a sequence $\left\{ \varphi_{i} \right\} \subset \calD(U;E)$ satisfying (i)-(iii) by Lemma \ref{lemmaaprox}. As a consequence of Fatou's Lemma and the previous estimate we conclude (extracting if necessary a subsequence) that:
\begin{equation}\nonumber
\|g\|_{W^{\nu-1,{N/N-1}}}\leq \liminfe_{i \rightarrow \infty}\|\varphi_{i}\|_{W^{\nu-1,{N/N-1}}}\leq C  \lim_{i \rightarrow \infty}\| A(\cdot,D)\varphi_{i}\|_{{1}} 
=C \|D_{\mathcal{A}}g\|.
\end{equation}
 The proof is complete.
\end{proof}

\section{$\bsA$-charges and their extensions to $BV_{\bsA,c}$}\label{sec2}

%\begin{definition}
%The distribution $F:\Omega \rightarrow \Lambda^{n}\R^{n}$ is a $\bsL$-strong charge if for every $\epsilon>0$ and $R>0$, there exists $C>0$ such that if $\varphi \in C_{c}^{\infty}(B_{R})$,
%\begin{equation}\label{lstrong}
%\left| \int_{\Omega} F \wedge  \varphi \right| \leq C\|\varphi\|_{L^{1}}+\epsilon\|\nabla_{\bsL} \varphi\|_{L^{1}}.
%\end{equation}
%\end{definition}

We now get back to the original problem of finding, locally, a continuous solution to (\ref{diverl}). 
\subsection{$\bsA$-fluxes and $\bsA$-charges} Distributions which allow, in an open set $\Omega$, to solve continuously (\ref{diverl}), will be called \textit{$\bsA$-fluxes}.

\begin{definition}
A distribution $\mathcal{F} \in \calD'{(\Omega)}$ is called an \textit{$\bsA$-flux in $\Omega$} if the equation \eqref{diverl} % $\diver_{\bsL^{\ast}}v=F$
has a continuous solution in $\Omega$, \textit{i.e.} if there exists $v \in C(\Omega,F^*)$ such that one has, for all $\varphi\in\calD(\Omega;E)$:
\begin{equation}\label{eq.flux}
\mathcal{F}(\varphi)=\int_{\Omega}\bar{v}\cdot {A(\cdot,D)\varphi}, \quad \forall\; \varphi \in \calD(\Omega;E).
\end{equation}
\end{definition}

$\bsA$-fluxes satisfy the following continuity condition.
\begin{lemma}
If $\mathcal{F}$ is an $\bsA$-flux then $\lim_{i} \mathcal{F}(\varphi_{i})=0$ for every sequence $( \varphi_{i})_{i}\subseteq \calD(\Omega;E)$ verifying $\varphi_i\twoheadrightarrow 0$.
\end{lemma}
\begin{proof}
Let $\mathcal{F}$ be an $\bsA$-flux and let $v\in C(\Omega,F^*)$ be such that (\ref{eq.flux}) holds.
Fix a sequence $(\varphi_i)_{i}\subseteq\calD(\Omega;E)$ verifying $\varphi_i\twoheadrightarrow 0$, let $c:=\sup_i\|A\varphi_i\|_1<+\infty$ and choose a compact set $K\subset\subset\Omega$ for which one has $\supp\varphi_i\subseteq K$ for all $i$.

Fix now $\epsilon>0$. According to Weierstrass' approximation theorem, choose a vector field $w\in \calD(\Omega,F^*)$ for which one has $\sup_K |v-w|\leq\epsilon$ and compute, for all $i$:
$$
\left| \mathcal{F}(\varphi_i)\right|\leq \left|\int_\Omega (\bar{v}-\bar{w})\cdot {A(\cdot,D) \varphi_i}\right|+
\left|\int_\Omega \bar{w}\cdot {A(\cdot,D) \varphi_i}\right|\leq\epsilon \|A\varphi_i\|_1+\left| \int_\Omega {\varphi_i} \,\overline{A^*w}\right|\leq c\epsilon+ \|A^* w\|_\infty \|\varphi_i\|_1.
$$
We hence get $\limsupe_i |\mathcal{F}(\varphi_i)|\leq c\epsilon$,
and the result follows for $\epsilon>0$ is arbitrary.
\end{proof}

The above property suggests the following definition of linear functionals {enjoying some continuity property involving the operator $A$}.
\begin{definition}
A linear functional $\mathcal{F}: \calD(\Omega;E) \rightarrow \mathbb{C}$ is called an $\bsA$-charge in $\Omega$  if $\lim_{i} \mathcal{F}(\varphi_{i})=0$ for every sequence $( \varphi_{i})_{i}\subseteq \calD(\Omega;E)$ satisfying $\varphi_i\twoheadrightarrow 0$. The linear space of all $\bsA$-charges in $\Omega$ is denoted by $CH_{\bsA}(\Omega)$.
\end{definition}

The following characterization of $\bsA$-charges will be useful in the sequel.
\begin{proposition}
If $\mathcal{F}:\calD(\Omega;E) \rightarrow \mathbb{C}$ is a linear functional, then the following properties are equivalent
\begin{enumerate}
\item[(i)] $\fF$ is an $\bsA$-charge,
\item[(ii)] for every $\epsilon>0$ and each compact set $K \subset \subset \Omega$ there exists $\theta>0$ such that, for any $\varphi\in\calD_K(\Omega;E)$, one has:
\begin{equation}\label{eq.ch}
|\fF(\varphi)|\leq \theta\|\varphi\|_{W^{\nu-1,1}}+\epsilon\|A(\cdot,D)\varphi\|_1.
\end{equation}
\end{enumerate}
\end{proposition}
\begin{proof}
We proceed as in \cite[Proposition~2.6]{PP}.\

Since (ii) implies trivially (i), it suffices to show that the converse implication holds. To that purpose, assume (i) holds, \textit{i.e.} suppose that $\fF$ is an $\bsA$-charge. Fix $\epsilon>0$ and a compact set $K \subset \subset \Omega$. By hypothesis, there exists $\eta>0$ such that for every $\varphi \in \calD_K(\Omega;E)$ satisfying $\| \varphi \|_{W^{\nu-1,1}}\leq \eta$ and $\|A\varphi\|_{{1}}\leq 1$, we have $|\fF(\varphi)|\leq \epsilon$. We now define $\theta :=\epsilon / \eta$.

Fix $\varphi\in\calD_K(\Omega;E)$ and assume by homogeneity that one has $\|A(\cdot,D)\varphi\|_1=1$. If moreover one has $\|\varphi\|_{W^{\nu-1,1}}\leq \eta$, then one computes $|\fF(\varphi)|\leq \epsilon=\epsilon \|A(\cdot,D)\varphi\|_1$. If on the contrary we have $\|g\|_{W^{\nu-1,1}}> \eta$, we define $\tilde{\varphi}=\varphi \eta/\|\varphi\|_{W^{\nu-1,1}}$. We then have $\|\tilde{\varphi}\|_{W^{\nu-1,1}}= \eta$ as well as $\|A(\cdot,D)\tilde{\varphi}\|_{{1}}< 1$, and hence also $|\fF(\tilde{\varphi})|\leq \epsilon$; this yields finally $|\fF(\varphi)|=\|\varphi\|_{W^{\nu-1,1}}|f(\tilde{\varphi})|/\eta\leq \epsilon \|\varphi\|_{W^{\nu-1,1}}/\eta=\theta \|\varphi\|_{W^{\nu-1,1}}$.
\end{proof}

As we shall see now, $\bsA$-charges can be extended in a unique way to linear forms on $BV_{\bsA,c}$.
\begin{Proposition}
An $\bsA$-charge $\fF$ in $\Omega$ extends in a unique way to a linear functional $\widetilde{\fF}:BV_{\bsA,c}(\Omega)\to\bbC$ satisfying the following property: for any $\epsilon>0$ and each compact set $K\subset\subset\Omega$, there exists $\theta>0$ such that for any $g\in BV_{\bsA,K}(\Omega)$ one has:
\begin{equation}\label{eq.ch2}
|\widetilde{\fF}(g)|\leq\theta\|g\|_{W^{\nu-1,1}}+\epsilon\|D_\bsA g\|.
\end{equation}
\end{Proposition}
\begin{proof}
Given $g\in BV_{\bsA,c}(\Omega)$, fix $(\varphi_i)_{i}\subseteq\calD(\Omega;E)$ satisfying $\varphi_i\twoheadrightarrow g$ and observe that it follows from (\ref{eq.ch}) that $(\fF(\varphi_i))_i$ is a Cauchy sequence of complex numbers whose limit does not depend on the choice of sequence $(\varphi_i)\subseteq\calD(\Omega;E)$ satisfying $\varphi_i\twoheadrightarrow g$. We hence define $\widetilde{\fF}(g):=\lim_i \fF(\varphi_i)$. It now follows readily from (\ref{eq.ch}) and {Remark~\ref{rmk.sci}} that $\widetilde{\fF}$ satisfies the desired property.
\end{proof}
\begin{Remark}\label{rmk.conv-ch}
If $\widetilde{\fF}:BV_{\bsA,c}(\Omega)\to\bbC$ extends the $\bsA$-charge $\fF$, it is easy to see from the previous proposition that for any compactly supported sequence $(g_i)_{i}\subseteq BV_{\bsA,c}(\Omega)$ satisfying $g_i\to 0$, $i\to\infty$ in $W^{\nu-1,1}(\Omega;E)$ and $\sup_i \|D_\bsA g_i\|<+\infty$, one has $\fF(g_i)\to 0$, $i\to\infty$.
\end{Remark}
From now on, we shall identify any $\bsA$-charge with its extension to $BV_{\bsA,c}$ and use the same notation for the two linear forms.

\subsection{Examples of $\bsA$-charges} Let us define two important classes of $\bsA$-charges.
\begin{Example}\label{ex.flux}
In case $\fF$ is the $\bsA$-flux associated to $v\in C(\Omega,F^*)$ according to (\ref{eq.flux}), its unique extension to $BV_{\bsA,c}(\Omega)$ is the $\bsA$-charge:
$$
\Gamma(v):BV_{\bsA, c}(\Omega)\to\bbC, g\mapsto \int_{\Omega} \bar{v}\cdot d\left[{D_\bsA g}\right].
$$

To see this, fix $g\in BV_{\bsA,c}(\Omega)$ together with a sequence $(\varphi_i)_{i}\subseteq\calD(\Omega;E)$ satisfying $\varphi_i\twoheadrightarrow g$ and choose a compact set $K$ satisfying $\supp g\subseteq K\subset\subset \Omega$ as well as $\supp \varphi_i\subseteq K$ for all $i$.
Given $\epsilon>0$, choose $w\in \calD(\Omega,F^*)$ a smooth vector field satisfying $\sup_K |v-w|\leq\epsilon$ and compute:
$$
\left|\Gamma(v)(g)-\int_\Omega \bar{v}\cdot d\left[{D_\bsA g}\right]\right|=\lim_i \left| \int_\Omega \bar{v}\cdot {A(\cdot,D) \varphi_i}-\int_\Omega \bar{v}\cdot d\left[{D_\bsA g}\right]\right|.
$$
On the other hand we have for all $i$:
\begin{multline*}
\left| \int_\Omega \bar{v}\cdot {A(\cdot,D) \varphi_i}-\int_\Omega \bar{v}\cdot d\left[ {D_\bsA g}\right]\right|
\leq \left| \int_\Omega (\bar{v}-\bar{w})\cdot {A(\cdot,D) \varphi_i}\right|+ \left| \int_\Omega (\bar{v}-\bar{w})\cdot d\left[ {D_\bsA g}\right]\right|\\
+\left| \int_\Omega \bar{w}\cdot {A(\cdot,D) \varphi_i}-\int_\Omega \bar{w}\cdot d\left[ {D_\bsA g}\right]\right|
\leq \epsilon \|A\varphi_i\|_1+\epsilon \|D_\bsA g\| +\left| \int_\Omega {\varphi_i}\, \overline{A^* w}-\int_\Omega \bar{w}\cdot d\left[ {D_\bsL g}\right]\right|.
\end{multline*}
Using the properties of $(\varphi_i)_{i}$ and Lebesgue's dominated convergence theorem, we thus get:
$$
\lim_i \left| \int_\Omega \bar{v}\cdot {A(\cdot,D) \varphi_i}-\int_\Omega \bar{v}\cdot d\left[ {D_\bsA g}\right]\right|\leq 2\epsilon \|D_\bsA g\|+\left| \int_\Omega {g}\,\overline{ A^*(\cdot,D)w}-\int_\Omega \bar{w}\cdot d\left[ {D_\bsL g}\right]\right|=2\epsilon \|D_\bsA g\|,
$$
according to (\ref{eq.def-mes}). The result follows, for $\epsilon>0$ is arbitrary.
\end{Example}
\begin{Example}\label{ex.sigma}
Assume that $U$ supports a SGN inequality of type (\ref{dl}) for $BV_{\bsA}$ functions in $U$. Then given $f\in W^{-(\nu-1),N}_{\mathrm{loc}}(U;{E^*})$, $f$ extends uniquely to an $\bsA$-charge in $U$.

To see this, fix $K\subset\subset U$ and infer from Lemma~\ref{lemma.hb} that there exist $(g_{\alpha})_{|\alpha|=\nu-1} \subseteq L^{N}(U;{E^*})$ such that 
$$
\la f,\varphi\ra=\sum_{|\alpha|=\nu-1} \int_U {g_{\alpha}}\overline{\partial^{\alpha}\varphi}, \quad \forall \varphi \in  \calD_K(U;{E}).
$$
Now fix $\epsilon>0$ and choose $\theta>0$ large enough for $ \left( \int_{K\cap \{|g_{\alpha}|>\theta\}} |g_{\alpha}|^N \right)^{1/N} \leq \epsilon/C$ to hold for any $|\alpha|=\nu-1$, where $C$ is a positive constant satisfying \eqref{dl}. We then compute, for $\varphi\in\calD_K(U;{E})$:
\begin{eqnarray*}
|\la f,\varphi\ra|&\leq &{\theta \left( \sum_{|\alpha|=\nu-1}\int_{K\cap\{|g_{\alpha}|\leq \theta\}}|\partial^{\alpha}\varphi| \right)+ \sum_{|\alpha|=\nu-1} \int_{K\cap\{|g_{\alpha}|> \theta\}} |g_{\alpha}\partial^{\alpha}\varphi|},\\
&\leq &{\theta \|\varphi\|_{W^{\nu-1,1}}+\sum_{|\alpha|=\nu-1} \left(\int_{K\cap \{|g_{\alpha}|>\theta\}} |g_{\alpha}|^N\right)^{ 1/N}\|\partial^{\alpha}\varphi\|_{{N/N-1}}},\\
&\leq &{\theta \|\varphi\|_{W^{\nu-1,1}}+\epsilon\|A(\cdot, D)\varphi\|_{L^{1}}}.
\end{eqnarray*}
The conclusion that $f$ extends to an $\bsA$-charge follows by approximation.
\end{Example}
\begin{Example}\label{ex.lambda}
Assume that $U$ supports a SGN inequality of type (\ref{dl}) for $BV_{\bsA}$ functions in $U$. Given $f\in L^N_{\mathrm{loc}}(U;{E^*})$ we know from Example~\ref{ex.sob-neg} that it defines a distribution in $W^{-(\nu-1),N}_{\mathrm{loc}}(U;{E^*})$. We then define, for $\varphi\in\calD(U;{E})$:
$$
\la \Lambda(f),\varphi\ra:=\overline{\la f,\bar{\varphi}\ra}=\int_U \bar{f}\varphi.
$$
It follows from the previous example and from approximation that $\Lambda(f)$ extends uniquely to an $\bsA$-charge in $U$ verifying, for all $g\in BV_{A,c}(U)$:
$$
\la\Lambda(f),g\ra=\int_U \bar{f}g.
$$
\end{Example}
\begin{Remark}\label{rmk.ODE}
It is easy to see that for any $x_0\in\Omega$, there exists an open neighborhood  $U\subseteq\Omega$ of $x_{0}$ such that one has $\Lambda[\calD(U;E)]\subseteq\Gamma[C^\infty(U,F^*)]$.
{Given $\varphi\in\calD(U;E)$, thanks to the local solvability of the elliptic operator $\Delta_{\bsA}=A^{\ast}(\cdot,D)\circ A(\cdot,D)$ (see \cite[Corollary 4.8]{GS}), there exists $u\in C^\infty(U;E)$ a smooth solution to $\Delta_\bsA u=\varphi$ in $U$.  Let $v:=A(\cdot,D) u$.}
This yields, for any $g\in BV_{\bsA,c}(U)$:
$$
\Lambda(\varphi)(g)=\int_U \bar{\varphi} {g}=\int_U g \,\overline{A^{*}(\cdot,D) v}=\int_U \bar{v}\cdot d\left[{D_\bsA g}\right]=\Gamma(v)(g),
$$
for we could, in the computation above, replace $v$ by $v\chi$ where $\chi\in\calD(U)$ satisfies $\chi=1$ in a neighborhood of $\supp g$.
\end{Remark}
It turns out that a linear functional on $BV_{\bsA,c}$ is an $\bsA$-charge if and only if it is continuous with respect to some locally convex topology on $BV_{\bsA,c}$.
\subsection{Another characterization of $\bsA$-charges}
{In the sequel, a \textit{locally convex space} means a Hausdorff locally convex topological vector space. For any family  $\mathcal{O}$ of sets and any set $X$ we denote $\mathcal{O}  \lefthalfcup  X :=\left\{ O \cap X: O \in \mathcal{O} \right\}$. Following \cite[Theorem~3.3]{PMP} we define the following topology on $BV_{\bsA,c}(\Omega)$, called the \textit{localized topology} associated to the family of subspaces $BV_{\bsA,K,\lambda}(\Omega)$.
\begin{definition}
Let $\bsT_{\bsA}$ be the unique locally convex topology on $BV_{\bsA,c}(\Omega)$ such that
\begin{enumerate}
\item[(a)] $\bsT_{\bsA} \lefthalfcup BV_{\bsA,K,\lambda} \subseteq \bsT_{W^{\nu-1,1}} \lefthalfcup BV_{\bsA,K,\lambda}$ for all $K \subset \subset \Omega$ and $\lambda>0$ where we let:
$$BV_{\bsA,K,\lambda}=\left\{ g \in BV_{\bsA,c}(\Omega):\;\text{supp}\;g\subseteq K, \|D_{\bsA}g\|\leq \lambda \right\},$$
and where $\bsT_{W^{\nu-1,1}}$ is the $W^{\nu-1,1}$-topology;
\item[(b)] for every locally convex space $Y$, a linear map $f:(BV_{\bsA,c};\bsT_{\bsA}) \rightarrow Y$ is continuous if only if $ f\upharpoonright{BV_{\bsA,K,\lambda}}$ is $W^{\nu-1,1}$ continuous for all $K \subset \subset \Omega$ and $\lambda>0$ .
\end{enumerate}
\end{definition}
\begin{Remark}
Uniqueness of the above topology is easily seen according to property (b). Concerning the existence, one can define the topology $\bsT_{\bsA}$ by constructing a basis of neighborhoods $\bsB_{\bsA}$ of the origin in the following way: denote by $\bsB_{\bsA}$ the collection of all absorbing, balanced and convex subsets $U\subseteq BV_{\bsA}(\Omega)$ satisfying $U\lefthalfcup BV_{\bsA,K,\lambda}\in \bsT_{W^{\nu-1,1}} \lefthalfcup BV_{\bsA,K,\lambda}$. Calling $\bsT_{\bsA}$ the vector topology on $BV_{\bsA}(\Omega)$ admitting $\bsB_{\bsA}$ as a neighborhood basis at the origin, one can see that it satisfies properties (a) and (b) above.

Choosing $(K_k)_{k\in\N}$ an increasing sequence of compact sets exhausting $\Omega$ and defining $X_k:=BV_{\bsA, K_k}(\Omega)$ for all $k\in\N$, we have:
$$
BV_{\bsA, K_k,k}=\{g\in X_k: \|D_{\bsA} g\|\leq k\}.
$$
Since it is straightforward to see that all the vector spaces $BV_{\bsA,K_k}$ are closed in the $W^{\nu-1,1}$ topology, and that $\|D_\bsA\cdot \|$ is a lower semicontinuous norm on $BV_{A,c}(\Omega)$, it now follows readily from \cite[Proposition~3.11]{PMP} that $\bsT_{\bsA}$-continuous linear functionals on $BV_{\bsA}(\Omega)$ are exactly the $\bsA$-charges in $\Omega$.
\end{Remark}
\begin{Proposition}
A linear functional $\fF:BV_{\bsA,c}(\Omega)\to\bbC$ is an $\bsA$-charge if and only if it is $\bsT_\bsA$-continuous.
\end{Proposition}

The following result will be useful in the sequel.
\begin{Corollary}\label{cor.compacite}
Assume that $K\subset\subset\R^n$ is a compact set and that $\lambda>0$ is a real number. If $(g_i)_{i\in I}\subseteq BV_{\bsA,K,\lambda}(U)$ converges to $0$ as distributions, \textit{i.e.} if one has $\int_U g_i\varphi\to 0$ for all $\varphi\in\calD(U;E^*)$, then the net $(\|g_i\|_{W^{\nu-1,1}})_{i\in I}$ also converges to $0$.
\end{Corollary}
\begin{proof}
Proceed towards a contradiction, assume that $(g_i)_{i\in I}$ is as in the statement, and that $(\|g_i\|_{W^{\nu-1,1}})_{i\in I}$ fails to converge to $0$, meaning that there is $\epsilon>0$ such that for all $i\in I$, one can find $j\in J$ satisfying $j\geq i$ and $\|g_j\|_{W^{\nu-1,1}}\geq\epsilon$. Define then $J:=\{i\in I: \|g_j\|_{W^{\nu-1,1}}\geq \epsilon\}$, observe $J$ is a directed set and consider the net $(g_j)_{j\in J}$. Since $BV_{\bsA,K,\lambda}(U)$ is compact in the $W^{\nu-1,1}$ topology according to Proposition~\ref{prop.compacite}, we know that there exists a cluster point $g\in BV_{\bsA,K,\lambda}(U)$ of $(g_j)_{j\in J}$ in the $W^{\nu-1,1}$ topology. It's easy to see from the definition of $J$ that one has $g\neq 0$. On the other hand, fix $\varphi\in\calD(U;{E^*})$. Since $(g_j)_{j\in J}$ converges to $0$ as distributions, we get for $j\in J$:
$$
\int_U g_j\varphi\to 0.
$$
Yet we should also get for $j\in J$:
$$
\left|\int_U g_j\varphi-\int_U g\varphi\right|\leq C_{\varphi} \|g_j-g\|_{W^{\nu-1,1}(U)}\to 0,
$$
which implies that $\int_U g\varphi=0$. Since $\varphi\in\calD(U;{E^*})$ is arbitrary, this means that $g=0$, which is a contradiction; the proof is complete.
\end{proof}

We now turn to proving the key result for obtaining Theorem~\ref{mainthm}.}
\section{Towards Theorem~\ref{mainthm}}\label{sec3}

Throughout this section, we assume that the open set $U\subseteq\Omega$ supports inequalities of type \eqref{hp3ineq} and \eqref{jh}; we also assume that one has $\Lambda[\calD(U;{E})]\subseteq \Gamma[C(U,{F^*})]$.
\begin{Remark}
It follows from Theorem~\ref{thm.hp}, Proposition~\ref{prop.gradmod} and Remark~\ref{rmk.ODE} that for any $x_0\in \Omega$, one can find an open neigborhood $U$ of $x_0$ in $\Omega$ satisfying all the above assumptions.
\end{Remark}

Our intention is to prove the following result.

\begin{Theorem}
If $\fF:BV_{\bsA,c}(U)\to F$ is an $\bsA$-charge in $U$, then there exists $v\in C(U, {F^*})$ for which one has $\fF=\Gamma(v)$, \textit{i.e.} such that one has, for any $g\in BV_{\bsA,c}(U)$:
$$
\fF(g)=\int_U \bar{v}\cdot d\left[{D_A g}\right].
$$
\end{Theorem}
To prove this theorem, we have to show that the map
$$
\Gamma: C( U , {F^*}) \longrightarrow CH_{\bsA}( U ), v\mapsto \Gamma(v),
$$
is surjective. In order to do this, we endow $C(U,{F^*})$ with the usual Fr\'echet topology of uniform convergence on compact sets, and $CH_\bsA ( U )$ with the Fr\'echet topology associated to the family of seminorms $(\|\cdot\|_K)_K$ defined by:
$$\|\fF\|_{K} := \sup\left\{ |\fF(g) |: g \in BV_{\bsA,K}( U ), \|D_\bsA g\|\leq 1\right\},$$
where $K$ ranges over all compact sets $K\subset\subset U $. The surjectivity of $\Gamma$ will be proven in case we show that $\Gamma$ is continuous and verifies the following two facts:
\begin{enumerate}
\item[(a)] $\Gamma[C( U ,{F^*})]$ is dense in $CH_{\bsA}( U )$.
\item[(b)] $\Gamma^{\ast}[CH_{\bsA}( U )^{\ast}]$ is sequentially closed in the strong topology of $C( U ,{ F^*})^*$.
\end{enumerate}
Indeed, it will then follow from the Closed Range Theorem \cite[Theorem~8.6.13]{EDW} together with\cite[Proposition~6.8]{PMP} and (b) that $\Gamma[C( U ,{F^*})]$ is closed in $CH_{\bsA}( U )$. Using (a) we shall then conclude that one has:$$\Gamma[C( U ,{F^*})]=CH_{\bsA}( U ),$$
\textit{i.e.} that $\Gamma$ is surjective.

The strategy of the proof of (b) follows the lines of De Pauw and Pfeffer's proof in \cite{PP}. For the proof of (a), however, the proof presented below is slightly different from their approach; we namely manage to avoid an explicit smoothing process and choose instead to use an abstract approach similar to the one used in \cite{M} in order to solve the equation $d\omega=F$.

Let us start by showing that $\Gamma: C( U ,{F^*}) \longrightarrow CH_{\bsA}( U )$ is linear and continuous. Indeed given a compact set $K\subset\subset  U $ and $g \in BV_{\bsA,K}( U )$ we have:
$$ |\Gamma(v)(g) | = \left|  \int_{ U } \bar{v}\cdot d\left[{D_{\bsA}g}\right] \right| \leq \| D_{\bsA}g \|\|v\|_{\infty,K},   $$
which implies  $\|\Gamma(v)\|_{K}\leq \|v\|_{\infty,K}$.
% $\Gamma$ is continuous.
%\begin{lemma}
%The map $\Gamma: C( U ,F) \longrightarrow CH_{\bsA}( U )$ is linear and continuous.
%\end{lemma}
%\begin{proof}
%Indeed given a compact set $K\subset\subset  U $ and $g \in BV_{K,\bsA}( U )$ we have:
%$$ |\Gamma(v)(g) | = \left|  \int_{ U } \bar{v}\cdot d\left[{D_{\bsA}g}\right] \right| \leq \| D_{\bsA}g \|\|v\|_{\infty,K},   $$
%which implies  $\|\Gamma(v)\|_{K}\leq \|v\|_{\infty,K}$.
%\end{proof}
Next we identify the dual space $CH_{\bsA}( U )^{*}$.
\subsection{Identifying $CH_\bsA( U )^*$}
%The following result is the identification we need.
\begin{proposition}
The map $\Phi: BV_{\bsA,c}( U ) \longrightarrow CH_{\bsA}( U )^{*}$ given by $\Phi(g)(F):=F(g)$ is a linear bijection.
\end{proposition}

%The proof of the previous proposition is quite delicate. We shall proceed in several steps which will be interesting as such.\\

First let us check that $\Phi$ is well defined. In fact, given $K \subset \subset  U $ and $g\in BV_{\bsA,K}(U)$ we have
$$ \left|\Phi(g)(F) \right|= \left|  F(g) \right|\leq \|D_{\bsA}g\|\|F\|_{K},$$
according to the definition of $\|\cdot\|_{K}$. Hence $\Phi(g)$ is continuous and $\Phi(g) \in CH_{\bsA}( U )^{\ast}$. \\

To show that $\Phi$ is injective, let $g \in BV_{\bsA,c}( U )$ be such that $\Phi(g)=0$. Then for any $B \subset  U $ measurable and bounded and for any $e^*\in E^*$ we have:
\begin{equation}\nonumber
\int_{B}{e^*g}=\int_{ U }\chi_{B}{e^*g}=\Lambda(\chi_{B}e^*)(g)=\Phi(g)[\Lambda(\chi_{B}e^*)]=0.
\end{equation}
Thus $e^*g=0$ a.e. in $ U $, which implies that $g=0$ (and hence that $\Phi$ injective) since $e^*\in E^*$ is arbitrary.\\

The next step is to prove that $\Phi$ is surjective. To show this property we shall define a right inverse for $\Phi$, called $\Psi$.\\

Let $\Psi: CH_{\bsA}( U )^{\ast} \longrightarrow \calD' ( U; E^* )$ be defined by:
\begin{equation}
\Psi(\alpha)[\varphi] :=\alpha[\Lambda(\varphi)].
\end{equation}
We claim that $\Psi$ is well defined, \textit{i.e.} that for $\alpha \in CH_{\bsA}( U )^{\ast}$, we have $\Psi(\alpha) \in BV_{\bsA,c}( U )$. Indeed, given $\alpha \in CH_{\bsL}( U )^{\ast}$ there exist $C>0$ and $K \subset \subset  U $ such that for all $F \in CH_{\bsL}( U )$ we have $|\alpha(F)|\leq C\|F\|_{K}$. In particular, for every $\varphi \in \calD( U; E^* )$ we have:
\begin{multline*}
{|\Psi(\alpha)(\varphi)| }\leq C \|\Lambda(\varphi)\|_{K}\\\leq C{\sup \left\{ |\Lambda(\varphi)(g)| :  g \in BV_{\bsA,K}( U ),  \|D_{\bsA}g\|\leq 1  \right\}}
\leq { C \sup \left\{ \int_{ U }|\bar{\varphi}g|: g \in BV_{\bsA,K}( U ),  \|D_{\bsA}g\|\leq 1  \right\}},
\end{multline*}
from which it already follows that one has $\supp \Phi(\alpha)\subseteq K$, since the above inequalities yield $\Phi(\alpha)(\varphi)=0$ if $\supp\varphi\cap K=\emptyset$. According to Remark~\ref{rmk.sobolev}, we see moreover that for any $g\in BV_{\bsA, K}(U)$ satisfying $\|D_{\bsA} g\|\leq 1$, one has: 
\begin{align*}
\int_U |\bar{\varphi} g| & \leq \|\varphi\|_{\bW^{\nu-1,N/(N-1)}(U;E)^*} \|g\|_{\bW^{\nu-1,N/(N-1)}(U;E)}\\
&\leq C(K,\nu) \|\varphi\|_{\bW^{\nu-1,N/(N-1)}(U;E)^*} \|g\|_{W^{\nu-1,N/(N-1)}(U;E)} \\
&\leq \tilde{C}(K,\nu) \|\varphi\|_{\bW^{\nu-1,N/(N-1)}(U;E)^*},
\end{align*}
where the latter inequality comes from the SGN inequality in $BV_{\bsA}$ (Proposition~\ref{prop.SGNBV}). It follows then from the reflexivity of $\bW^{\nu-1, N/(N-1)}(U;E)$ that one has $g\in \bW^{\nu-1,N/(N-1)}(U;E)\subseteq W^{\nu-1,N/(N-1)}(U;E)$: using, indeed, Hahn-Banach's theorem we can extend the map $T:\calD(U;E^*)\to\bbC,\psi\mapsto \int_U g\psi$ into a map $\widetilde{T}\in [\bW^{\nu-1,N/(N-1)}(U;E)^*]^*$ itself being the evaluation at $h\in \bW^{\nu-1,N/(N-1)}(U;E)$, so that one has for $\varphi\in\calD(U;E^*)$:
$$
\int_U \bar{\varphi}g=T(\bar{\varphi})=\bar{T}(\bar{\varphi})=\la \bar{\varphi}, h\ra=\int_U \bar{\varphi} h.
$$
This yields $g=h$ a.e., meaning that one has $g\in\bW^{\nu-1, N/(N-1)}(U;E)$.

Moreover, for any $v \in C^{\infty}_{c}( U , F^*)$ we have:
\begin{eqnarray*}
 |\Psi(\alpha)[A^{\ast}(\cdot,D)v]|&=& |\alpha[\Lambda(A^{\ast}(\cdot,D)v)]|,\\ & \leq&{ C \|\Lambda(A^{\ast}(\cdot,D)v)\|_{K}}\\
&\leq & C{\sup \left\{ \left|\int_{ U }g\,\overline{A^{\ast}(\cdot,D)v}\right|: g \in BV_{\bsA,K}( U ),  \|D_{\bsA}g\|\leq 1  \right\}}\\
& \leq & C \sup \left\{ \|D_{\bsA}g\| \|v\|_{\infty}  : g \in BV_{\bsA,K}( U ),  \|D_{\bsA}g\|\leq 1  \right\}\\
&\leq  &C{\sup \|v\|_{\infty},  }
\end{eqnarray*}
so that one has $\Psi(\alpha) \in BV_{\bsA,c}( U )$.

\begin{lemma}\label{lem.inv} The maps $\Phi$ and $\Psi$ defined above are inverses, \textit{i.e.} we have:
\begin{enumerate}
\item[(i)] $\Psi \circ \Phi=Id_{BV_{\bsA,c}( U )}$;
\item[(ii)] $\Phi \circ \Psi=Id_{CH_{\bsA}( U )^{*}}$ (in particular, $\Phi$ is surjective).
\end{enumerate}
\end{lemma}
In order to prove the previous lemma, we shall need some observations concerning the polar sets of some neighborhoods of the origin in $CH_{\bsA}( U )$. First, observe that the family of all sets $V(K,\epsilon)$ (where $K$ ranges over all compact subsets of $ U $, and $\epsilon$ over all positive real numbers) defined by:
$$
V(K,\epsilon):=\{F\in CH_\bsA( U ):\|F\|_K\leq\epsilon\},
$$
is a basis of neighborhoods of the origin in $CH_\bsA ( U )$.
\begin{Claim}\label{cl.w*}
Fix $K\subset\subset U $ a compact set and a real number $\epsilon>0$. For any $\alpha\in V(K,\epsilon)^\circ$, one has:
\begin{itemize}
\item[(i)] $\supp\Psi(\alpha)\subseteq K$;
\item[(ii)] $\|D_\bsA \Psi(\alpha)\|\leq\frac{1}{\epsilon}$.
\end{itemize}
\end{Claim}
\begin{proof}
To prove (i), assume that $\varphi\in\calD(U)$ satisfies $K\cap \supp\varphi=\emptyset$. Then, we get for $\lambda>0$:
$$
\|\lambda\Lambda(\varphi)\|_K=\sup\left\{\lambda \left|\int_ U  \bar{\varphi} g\right|:g\in BV_{\bsA, K,\lambda}( U ), \|D_\bsA g\|\leq 1\right\}=0.
$$
In particular this yields $\lambda \Lambda(\varphi)\in V(K,\epsilon)$. We hence obtain:
$$
\lambda |\alpha[\Lambda(\varphi)]|=|\alpha[\lambda\Lambda(\varphi)]|\leq 1,
$$
for any $\lambda>0$. Since $\lambda>0$ is arbitrary, this implies that one has $\alpha[\Lambda(\varphi)]=0$, \textit{i.e.} that $\Psi(\alpha)(\varphi)=0$. We may now conclude that $\supp\Psi(\alpha)\subseteq K$.
In order to obtain statement (ii), fix $v\in C^\infty_c( U ,F^*)$ satisfying $\|v\|_{\infty}\leq 1$ and compute:
\begin{eqnarray*}
\|\epsilon \Lambda[A^*(\cdot, D) v]\|_K&=&\epsilon \|\Lambda[{A^*(\cdot, D)v}]\|_K,\\&=&\epsilon\sup\left\{\left|\int_ U  {g}\,{\overline{A^*(\cdot, D) v}}\right|:g\in BV_{\bsA,K},\|D_\bsA g\|\leq 1\right\},\\
&=&\epsilon\sup\left\{\left|\int_U \bar{v}\cdot d[D_\bsA {g}]\right|: g\in BV_{\bsA, K}, \|D_\bsA g\|\leq 1\right\},\\
&\leq &\epsilon \sup\{\|D_\bsA g\|\cdot \|v\|_\infty:g\in BV_{\bsA, K}, \|D_\bsA g\|\leq 1\},\\
&\leq &\epsilon,
\end{eqnarray*}
so that one has $\epsilon\Lambda({A^*(\cdot, D)v})\in V(K,\epsilon)$. It hence follows that:
$$
\epsilon |\Psi(\alpha)(A^*(\cdot, D)v)|=|\alpha[\epsilon\Lambda({A^*(\cdot, D) v})]|\leq 1,
$$
and we thus get:
$$
|\Psi(\alpha)[A^*(\cdot, D)v]|\leq\frac{1}{\epsilon}.
$$
Since $v\in \calD( U ,F^*)$ is an arbitrary vector field satisfying $\|v\|_\infty\leq 1$, this yields $\|D_\bsA \Psi(\alpha)\|\leq\frac{1}{\epsilon}$, and concludes the proof of the claim.
\end{proof}

We now turn to proving Lemma~\ref{lem.inv}.
\begin{proof}[Proof of Lemma~\ref{lem.inv}]
To prove part (i), fix $g\in BV_{\bsA,c}( U )$ and compute, for $\varphi\in\calD( U {; E^*})$:
$$
\Psi[\Phi(g)](\varphi):=\Phi(g)[\Lambda({\varphi})]=\Lambda({\varphi})({g})=\int_ U  \bar{\varphi}g,
$$
that is, $\Psi[\Phi(g)]=g$ in the sense of distributions.

In order to prove part (ii), fix $\alpha\in CH_{\bsA}^*( U )$. We have to show that, for any $F\in CH_\bsA( U )$, we have:
$$
\Phi[\Psi(\alpha)](F)=\alpha(F),
$$
\textit{i.e.} that for any $F\in CH_\bsA( U )$, one has:
$$
F[\Psi(\alpha)]=\alpha(F).
$$
To this purpose, define for any $F\in CH_\bsA( U )$ a map:
$$
\Delta_F:CH_\bsA( U )^*\to\bbC, \alpha\mapsto \Delta_F(\alpha):=F[\Psi(\alpha)].
$$

\begin{Claim}\label{cl.bg}
Given $F\in CH_\bsA( U )$, the map $\Delta_F$ is weakly$^*$-continuous on $V(K,\epsilon)^\circ$ for all $K\subset\subset U $ and $\epsilon>0$.
\end{Claim}
To prove this claim, fix $K\subset\subset  U $, $\epsilon>0$ and assume that $(\alpha_i)_{i\in I}\subseteq V(K,\epsilon)^\circ$ is a net weak$^*$-converging to $0$. In particular one gets:
\begin{itemize}
\item[(a)] for any $\varphi\in\calD( U {; E^*})$, we have $\Lambda(\varphi)\in CH_\bsA( U )$ and hence the net $(\Psi(\alpha_i)(\varphi))_{i\in I}=(\alpha_i[\Lambda(\varphi)])_{i\in I}$ converges to $0$; in other terms, the net $(\Psi(\alpha_i))_{i\in I}$ converges to $0$ in the sense of distributions. 
\end{itemize}
According to Claim~\ref{cl.w*}, we moreover have:
\begin{itemize}
\item[(b)] $\supp\Psi(\alpha_i)\subseteq K$ for each $i\in I$;
\item[(c)] $c:=\sup_{i\in I} \|D_\bsA \Psi(\alpha_i)\|\leq\frac{1}{\epsilon}$.
\end{itemize}
We thus have $(\Psi(\alpha_i))_{i\in I}\subseteq BV_{\bsA, K, 1/\epsilon}$.
It hence follow from Corollary~\ref{cor.compacite} that the net $(\|\Psi(\alpha_i)\|_{W^{\nu-1,1}})_{i\in I}$ converges to $0$. From the fact that $F$ is an $\bsA$-charge we see that the net $(F[\Psi(\alpha_i)])_{i\in I}$ converges to $0$  as well. This means, in turn, that $(\Delta_F(\alpha_i))_{i\in I}$ converges to $0$, which shows that $\Delta_F$ is weak$^*$-continuous on $V(K,\epsilon)^\circ$.

\begin{Claim}
For any $\alpha\in CH_\bsA( U )^*$, we have $\Delta_F(\alpha)=\alpha(F)$.
\end{Claim}
To prove the latter claim, observe that according to Claim~\ref{cl.bg} and to the Banach-Grothendieck theorem \cite[Theorem~8.5.1]{EDW}, there exists $\tilde{F}\in CH_\bsA( U )$ such that for any $\alpha\in CH_\bsA( U )^*$, we have:
$$
\Delta_F(\alpha)=\alpha(\tilde{F}).
$$
Yet given $g\in BV_{\bsA,c}( U )$, we then have, according to [Lemma~\ref{lem.inv}, (i)]:
$$
F(g)=F\{\Psi[\Phi(g)]\}=\Delta_F[\Phi(g)]=\Phi(g)(\tilde{F})=\tilde{F}(g),
$$
\textit{i.e.} $F=\tilde{F}$, which proves the claim.

It now suffices to observe that Lemma~\ref{lem.inv} is proven for we have established the equality $F[\Psi(\alpha)]=\alpha(F)$ for any $F\in CH_\bsA( U )$ and $\alpha\in CH_\bsA( U )^*$.
\end{proof}

As a corollary, we get a proof of the density of $\Lambda[\calD(U)]$ and $\Gamma[C( U ,F^*)]$ in $CH_\bsA( U )$.
\begin{Corollary}
The space $\Lambda[\calD( U{; E^*} )]$ is dense in $CH_\bsA( U )$.
\end{Corollary}
\begin{proof}
Assuming that $\alpha\in CH_\bsA( U )^*$ satisfies $\alpha\upharpoonright \Lambda[\calD( U; E^*)]=0$, we compute for any $\varphi\in\calD( U; E^* )$:
$$
\Psi(\alpha)(\varphi):=\alpha[\Lambda(\varphi)]=0.
$$
This means that $\Psi(\alpha)=0$, and implies that $\alpha=\Phi\circ\Psi (\alpha)=\Phi(0)=0$. The result then follows from the Hahn-Banach theorem.
\end{proof}
\begin{Corollary}\label{cor.dens-2}
The space $\Gamma[C( U ,F^*)]$ is dense in $CH_\bsA( U )$.
\end{Corollary}
\begin{proof}
It follows from the previous corollary that $\Lambda[\calD( U; E^*)]$ is dense in $CH_\bsA( U )$. Since by hypothesis we also have $\Lambda[\calD(U; E^*)]\subseteq \Gamma[C(U,F^*)]\subseteq CH_\bsA(U)$, it is clear that $\Gamma(U,F^*)$ is dense in $CH_\bsA(U)$.
\end{proof}

In order to study the range of $\Gamma^*$, we introduce the following linear operator:
$$
\Xi:BV_{\bsA,c}( U )\to C( U ,F^*)^*, g\mapsto \Xi(g),
$$
defined by $\Xi(g)(v):=\Gamma(v)(g)$ for any $v\in C( U ,F^*)$.

\begin{Claim}
We have $\im\Gamma^*=\im \Xi$.
\end{Claim}
\begin{proof}
To prove this claim, fix $\mu\in C( U ,F^*)$. If one has $\mu=\Gamma^*(\alpha)$ for some $\alpha\in CH_\bsA( U )^*$, then we compute for $v\in C( U ,F^*)$:
$$
\Xi[\Psi(\alpha)](v)=\Gamma(v)[\Psi(\alpha)]=\Phi[\Psi(\alpha)][\Gamma(v)]=\alpha[\Gamma(v)]=\Gamma^*(\alpha)(v)=\mu(v),
$$
so that one has $\mu=\Xi[\Psi(\alpha)]\in \im\Xi$. Conversely, if one has $\mu=\Xi(g)$ for some $g\in BV_{\bsA,c}( U )$, then we compute for $v\in C( U ,F^*)$:
$$
\Gamma^*[\Phi(g)](v)=\Phi(g)[\Gamma(v)]=\Gamma(v)(g)=\Xi(g)(v)=\mu(v),
$$
so that one has $\mu=\Gamma^*[\Phi(g)]\in\im\Gamma^*$.
\end{proof}

Consider the set
$$
B:=\{v\in C( U ,F^*): \|v\|_\infty\leq 1\}.
$$
It is clear that $B$ is bounded in $C( U ,F^*)$. Hence the seminorm:
$$
p:C( U ,F^*)^*\to \R_+, \mu\mapsto p(\mu):=\sup_{v\in B} |\mu(v)|,
$$
is strongly continuous (\emph{i.e.} continuous with respect to the strong topology) on $C( U ,F^*)^*$. Observe now that one has, for $g\in BV_{\bsA,c}( U )$:
\begin{eqnarray*}
p[\Xi(g)]&=&\sup_{v\in B} |\Xi(g)(v)|,\\
&=&\sup\{|\Gamma(v)(g)|:v\in B\},\\
&=& \|D_\bsA g\|.
\end{eqnarray*}
\begin{Lemma}\label{lem.closed}
The set $\im \Xi$ is strongly sequentially closed in $C( U ,F^*)^*$.
\end{Lemma}
\begin{proof}
Fix a sequence $(\Xi(g_k))_{k\in\N}\subseteq\im \Xi$ and assume that, in the strong topology, one has:
$$
\Xi(g_k)\to \mu\in C( U ,F^*)^*,\quad k\to\infty.
$$
The strong continuity of $p$ then yields:
$$
c:=\sup_{k\in\N} \|D_\bsA g_k\|=\sup_{k\in\N} p[\Xi(g_i)]<+\infty.
$$

\begin{Claim}
There exists a compact set $K\subset\subset U $ such that one has $\supp g_k\subseteq K$ for each $k\in\N$.
\end{Claim}
To prove this claim, let us first prove that the sequence $(\supp D_{\bsA} g_k)_{k\in\N}$ is compactly supported in $ U $ (\textit{i.e.} that there is a compact subset of $ U $ containing $\supp D_{\bsA}g_k$ for all $k$). To this purpose, we proceed towards a contradiction and assume that it is not the case. Let then $ U =\bigcup_{j\in\N}  U _j$ be an exhaustion of $ U $ by open sets satisfying, for each $j\in\N$, $\bar{ U }_j\subseteq  U _{j+1}$ and such that $\bar{ U }_j$ is a compact subset of $ U $ for each $j\in\N$. Since $(\supp D_\bsA g_k)_{k\in\N}$ is not compactly supported, there exist increasing sequences of integers $(j_l)_{l\in\N}$ and $(k_l)_{l\in\N}$ satisfying, for any $l\in\N$:
$$
\supp (D_\bsA g_{k_l})\cap ( U _{j_l+1}\setminus\bar{ U }_{j_l})\neq \emptyset.
$$
In particular, there exists for each $l\in\N$ a vector field $v_l\in C_c( U_{j_{l}+1}\setminus\bar{U}_{j_l} ,F^*)$ with $\|v_l\|_\infty\leq 1$ and:
$$
a_l:=\left|\int_ U  \bar{v}_l\cdot d[D_{\bsA}g_{k_l}]\right|>0.
$$
Let now, for $l\in\N$, $b_l:=\max_{0\leq k\leq l} \frac{1}{a_k}$ and define a bounded set $B'\subseteq C( U ,F^*)$ by:
$$
B':=\left\{v\in C( U ,F^*): \|v\|_{\infty,\bar{ U }_{j_l+1}} \leq l b_l\text{ for each }l\in\N\right\}.
$$
It follows from the construction of $B$ that one has $w_l:=lb_lv_l\in B$ for any $l\in\N$. Moreover the seminorm
$$
p':=C( U ,F^*)^*\to\R_+, \mu\mapsto \sup_{v\in B'} |\mu(v)|,
$$
is strongly continuous. Yet we get for $l\in\N$:
$$
p'[\Xi(g_{k_l})]\geq |\Xi(g_{k_l})(w_l)|=|\Gamma(w_l)(g_{k_l})|=lb_l \left|\int_{ U } \bar{v}_l\cdot d(Dg_{k_l})\right|=lb_la_l\geq l.
$$
Since this yields $p'[\Xi(g_{k_l})]\to\infty$, $l\to\infty$, we get a contradiction with the fact that $p'$ is strongly continuous (recall that $(\Xi(g_{k_l}))_{l\in\N}$ converges in the strong topology).

\begin{Claim} Let $V$ be an open set and let $r(x,D)\in S^{-\infty}$ be a regularizing operator. Assume that $g\in L^{N/N-1}(V)$ satisfies $\int_V g[\psi-r(x,D)\psi]=0$ for all $\psi\in \calD(V,F^*)$. Under those assumptions, one has $g=0$ in $V$.
\end{Claim}
\begin{proof}
Fix $x_0\in V$ and let $\ell>0$ be such that $B(x_0,\ell)\subseteq V$. Given $\psi\in \calD(B(x_0,\ell),F^*)$, begin by observing that, for all $\phi\in \calD(B(x_0,\ell), F)$, one has:
$$
\left|\int_{B(x_0,\ell)} \la \phi,r(x,D) \psi\ra\right|\leq C\|\phi\|_{L^1}\| \psi\|_{L^N}\leq C' \,\ell \, \|\phi\|_{L^{N/N-1}} \|\psi\|_{L^N},
$$
where the first inequality follows from \cite[inequality (3.3)]{HP1}. It hence follows by duality that one has:
$$
\|r(x,D) \psi\|_{L^N(B(x_0,\ell)}=\sup_{\begin{subarray}{c}\phi\in\calD(B(x_0,\ell), F)\\\|\phi\|_{L^{N/N-1}}\leq 1\end{subarray}} \int_{B(x_0,r)} \la \phi,r(x,D)\psi\ra\leq C\, \ell \, \|\psi\|_{L^N}.
$$
Assuming hence that $\ell$ is small enough (say $\ell=\ell_0$), this yields $\|r(x,D)\|_{L^N(B(x_0,\ell_0))\to L^N(B(x_0,\ell_0))}<1$.\\
Now writing $\psi-r(x,D)\psi=(I-r(x,D)) \psi$, and using \cite[Exercice~6.14]{BREZIS} to infer that $I-r(x,D)$ is a linear isomorphism of $L^N(B(x_0,\ell_0))$, we see that $[I-r(x,D)](\calD(B(x_0,\ell_0)))$ is dense in $L^N(B(x_0,\ell_0))$, which is sufficient to conclude that $g=0$ in $B(x_0,\ell_0)$, and hence that $g=0$ on $V$ since $x_0$ is arbitrary.
\end{proof}

Now choose $K\subset\subset U$ a compact set for which one has $\supp (D_\bsA g_k)\subseteq K$ for all $k\in\N$ and fix $k\in\N$.
Fix also $x_0\in U\setminus K$, choose $\ell>0$ so that $V:=B(x_0,\ell)\subseteq U$ and fix $f\in\calD(V,F^*)$. Define $u=A(\cdot, D)q(\cdot, D) f\in\calD(V,F^*)$ satisfying \eqref{pseudo}. One then has:
$$
\int_V {g}_k\overline{f+r(x, D) f}={\int_V g_k \overline{ A^*(x,D)u}}={\int_V \bar{u}\cdot d[D_\bsA g_k]}=0,
$$
since one knows that $V\cap \supp [D_\bsA g_k ]=\emptyset$. Applying the previous claim to $g_k$, $V$ and $-r(\cdot, D)$ we hence get $g_k=0$ on $V$.
It then follows that one has $\supp g_k\subseteq K$, for $x_0$ is an arbitrary point in $U\setminus K$.\\

Getting back to the proof of Lemma~\ref{lem.closed}, observe that, according to Proposition~\ref{prop.compacite}, there exists a subsequence $(g_{k_l})\subseteq (g_k)$, $W^{\nu-1,1}$-converging to $g\in BV_{\bsA,c}( U )$.
Using the fact that $\Gamma(v)$ is an $\bsA$-charge, we compute:
$$
\mu(v)=\lim_{l\to\infty} \Xi(g_{k_l})(v)=\lim_{l\to\infty} \Gamma(v)(g_{k_l})=\Gamma(v)(g)=\Xi(g)(v),
$$
and hence we get $\mu=\Xi(g)\in\im \Xi$.
\end{proof}

We hence proved the following theorem.
\begin{Theorem}
We have $CH_\bsA( U )=\Gamma[C( U ,F^*)]$.
\end{Theorem}

\section{Application: elliptic complexes of vector fields }\label{aplica}

Consider $n$ complex vector fields $L_{1},\dots,L_{n}$, $n \geq 1$, with smooth coefficients defined on an open set $\Omega \subset \R^{N}$, $N \geq 2$. Naturally, we assume that the vector fields $L_{j}$, $1\leq j\leq n$ do not vanish in  $\Omega$; in particular, they may be viewed as non-vanishing sections of the vector bundle $\mathbb{C} T(\Omega)$ as well as first order differential operators of principal type.
%Suppose now that $\bsL:=\left\{L_{1},\dots,L_{n}\right\}$ is a system of linearly independent vector fields and 

We impose two fundamental properties on those vector fields in our context; namely, we require that:
\begin{itemize}
  \item[(a)] $L_1,\dots,L_n$ are everywhere linearly independent;
  \item[(b)] the system $\bsL:=\{L_1,\dots,L_n\}$ is \textit{elliptic}.
\end{itemize} 
The latter means for any 1-form $\omega$, \emph{i.e.} any section of $T^{*}(\Omega)$, the equality $\langle \omega,L_{j} \rangle=0$ for $1 \leq j \leq n$ implies that one has $\omega=0$~---~which is equivalent to require that the second order operator:
%\begin{equation}\label{deltal}
$
\Delta_{\bsL}:=L_{1}^{*}L_{1}+...+L_{n}^{*}L_{n}%=\diver_{\bsL^{*}}\nabla_{\bsL}
$
%\end{equation}
is elliptic. We use the notation  $L_{j}^{*}\;:=\;\overline{ L_j^{t}}$ where $\overline L_{j}$ denotes the vector field obtained from $L_j$ by conjugating its coefficients and let $L_j^t$ denote the formal transpose of $L_j$ for $j=1,\dots,n$~---~namely this means that, for all (complex valued) $\varphi,\psi\in\calD(\Omega)$, we have:
 $$
 \int_\Omega (L_j\varphi)\bar{\psi}=\int_\Omega \varphi (\overline{L_j^*\psi}).
 $$
Consider  the gradient $\nabla_{\bsL}:C^{\infty}(\Omega)\longrightarrow C^{\infty}(\Omega)^n$ associated to the system $\bsL$ defined by 
$\nabla_{\bsL}\,u\,:= (L_{1}u,..,L_{n}u)$ %for $u \in C^{\infty}(\Omega)$ 
and its formal {complex}  adjoint operator, defined for $v \in C^{\infty}(\Omega,\mathbb{C}^{n})$ by:
\begin{equation}\nonumber
{\rm div}_{\bsL^{*}}\,v:=L^{*}_{1}v_{1}+...+L^{*}_{n}v_{n}.
\end{equation}
The following local continuous solvability result is known for divergence-type operators of the previous type; it is borrowed from \cite[Theorem 1.2]{MP}.

\begin{corollary}\label{teo6.1} Assume that the system of vector fields $\bsL$ satisfies (i) and (ii). Then every point $x_{0} \in \Omega$ is contained in an open neighborhood $U \subset \Omega$ such that for any $f \in \calD'({U})$, the equation:
%\begin{equation}\label{diverl}
$$\diver_{\bsL^{*}} v =f$$
%\end{equation}
is continuously solvable in $U$ if and only if $f$ is an $\bsL$-charge in $U$, meaning that for every $\epsilon>0$ and every compact set $K \subset \subset U$, there exists $\theta=\theta(K,\epsilon)>0$ such that one has, for every $\varphi\in\calD_K(U)$:
\begin{equation}\label{lstrong2}
\left| f(\varphi) \right| \leq {\theta}\|\varphi\|_{{1}}+\epsilon\|\nabla_{\bsL} \varphi\|_{{1}}.
\end{equation}
\end{corollary}

This result can be seen as a direct consequence of Theorem \ref{mainthm} applied to the first order operator $A(\cdot,D):=\nabla_\bsL$, which is elliptic and canceling. 
Indeed, from the fact that $\bsL$ is elliptic we easily see that $\nabla_\bsL$ is elliptic as well.
Furthermore, \cite[Lemma4.1]{HP3} together with the assumption that the system $\bsL$ be linearly independent, shows that $\nabla_\bsL$ is canceling.

Let  $C^{\infty}(\Omega,\Lambda^{k}\R^{n})$ denote the space of $k$-forms on $\R^{n}$, $0\leq k \leq n$, with smooth, complex coefficients defined on $\Omega$. Each $f \in C^{\infty}(\Omega,\Lambda^{k}\R^{n})$ may be written as :
\[
\displaystyle{f=\sum_{|I|=k}f_{I}dx_{I}}, \quad dx_{I}=dx_{i_1}\wedge\cdots\wedge dx_{i_k},
\]
where one has $f_{I}\in C^{\infty}(\Omega)$ and where $I=\left\{i_{1},...,i_{k}\right\}$ is a set of strictly increasing indices with $i_{l}\in \left\{1,...,n\right\}$,  $l=1,...,k$.  Consider the differential operators :
\[
d_{\bsL,k}: C^{\infty}(\Omega,\Lambda^{k}\R^{n})\rightarrow C^{\infty}(\Omega,\Lambda^{k+1}\R^{n})
\]
defined by:
%\begin{equation}\nonumber
$d_{\bsL,0}f:=\sum_{j=1}^{n}(L_{j}f)dx_{j}$ for $f \in C^{\infty}(\Omega)$,
%\end{equation}  
and, for $f=\sum_{|I|=k}f_{I}dx_{I}\in C^{\infty}(\Omega,\Lambda^{k}\R^{n})$, $1\le k\le n-1$, by:
\begin{equation}\nonumber
d_{\bsL,k}f:=\sum_{|I|=k}(d_{\bsL,0}f_{I})dx_{I}=\sum_{|I|=k}\sum_{j=1}^{n}(L_{j}f_{I})dx_{j} \wedge dx_{I}.
\end{equation}  
%Note that, in general, the complex property $d_{\L,k+1}d_{\L,k}=0$ does not hold. 
We also define the dual pseudo-complex $d^{*}_{\bsL,k}: C^{\infty}(\Omega,\Lambda^{k+1}\R^{n})\rightarrow 
C^{\infty}(\Omega,\Lambda^{k}\R^{n})$, $0 \leq k \leq n-1$, determined {by the following relation for any $u \in C_{c}^{\infty}(\Omega,\Lambda^{k}\R^{n})$ and $v \in C_{c}^{\infty}(\Omega,\Lambda^{k+1}\R^{n})$:}
$$\int d_{\bsL,k}u\cdot \overline{v}\; =\int u \cdot \overline{d^{*}_{\bsL,k}v},$$
where the dot indicates the standard pairing on forms of the same degree. This is to say that given $\displaystyle{f=\sum_{|J|=k}f_{J}dx_{J}}$, one has:
$$\displaystyle{d_{\bsL,k}^{*}f=\sum_{|J|=k}\sum_{j \in J}L_{j}^{*}f_{J}dx_{j} \vee dx_{J}},$$
where, for each  $j_{l} \in J=\left\{j_{1},...,j_{k}\right\}$ and  $l \in \left\{1,...,k\right\}$, $dx_{j_{l}} \vee dx_{J}$ is defined by:
$
dx_{j_{l}} \vee dx_{J}:=(-1)^{l+1}dx_{1} \wedge ... \wedge dx_{j_{l-1}}\wedge dx_{j_{l+1}}\wedge ... \wedge dx_{j_{k}}.
$

Suppose first that $\bsL$ is involutive, \textit{i.e.} that each commutator $[L_{j},L_{\ell}]$, $1\leq j,\ell \leq n$ is a linear combination of $L_{1},\dots,L_{n}$. 
%The subbundle $\mathbb{C} T (\Omega)$ spanned by $\bsL$, denoted by $(\Omega,\bsL)$, is called an involutive structure. 
Then the chain $\{ d_{\bsL,k} \}_{k}$  defines a complex of differential operators associated to the structure $\bsL$, which is precisely the \textit{de Rham complex} when $n=N$ and $L_{j}=\partial_{x_{j}}$  (see \cite{BCH} for more details). In the non-involutive situation, we do not get a complex in general, and the fundamental complex property $d_{\bsL,k+1}\circ d_{\bsL,k}=0$ might not hold. On the other hand, this chain still satisfies a ``pseudo-complex" property in the sense that
$d_{\bsL,k+1}\circ d_{\bsL,k}$ is a differential of operator of order one rather than two, as it is generically expected. We will refer to $(d_{\bsL,k},C^{\infty}(\Omega,\Lambda^{k}\R^{n}))$ as the pseudo-complex $\{d_{\bsL}\}$ associated with $\bsL$ on $\Omega$. 

Consider the operator $$A(\,\cdot \,,D)=(d_{\bsL,k}, d^{*}_{\bsL,k-1}): C_{c}^{\infty}(\Omega,\Lambda^{k}\R^{n}) \rightarrow C_{c}^{\infty}(\Omega,\Lambda^{k+1}\R^{n}) \times C_{c}^{\infty}(\Omega,\Lambda^{k-1}\R^{n}),$$
for $0\leq k \leq n$. Here the operator $d_{\bsL,-1}=d^*_{\bsL,-1} $ is understood to be zero.
The operator $A(\,\cdot \,,D)$ is ellipitic and canceling for $k\notin \{1,n-1\}$ (see Section 4 \cite{HP3} for details), so that for each $x_{0} \in \Omega$ there exists an neighborhood $U \subset \Omega$ of $x_{0}$ and $C>0$ such that the inequality:
\begin{equation}\nonumber
\|u\|_{L^{N/N-1}}\leq C (\|d_{\bsL,k}u\|_{L^{1}} + \|d^{*}_{\bsL,k-1}u\|_{L^{1}}),
\end{equation}
holds for any $u \in \calD(U,\Lambda^{k}\R^{n})$ (see \cite[Theorem B]{HP2}).
 
Now we consider the equation \ref{diverl} associated to operator $A(\,\cdot \,,D)$, \textit{i.e.} the equation:
\begin{equation}\label{aplic1}
d^{*}_{\bsL,k} u + d_{\bsL,k-1} v = f.
\end{equation}
The following local continuous solvability result for $\eqref{aplic1}$ is a consequence of our main theorem. 

\begin{corollary}\label{mainthm2}
Consider a system of complex vector fields $\bsL=\left\{ L_{1}, ... , L_{n} \right\}$, $n \geq 2$ satisfying hypotheses (i)-(ii) above, and the pseudo-complex $\{d_{\bsL,k}\}_{k}$ associated with $\bsL$ on $\Omega$ with $k\notin\{1,n-1\}$. Then every point $x_{0} \in \Omega$ is contained in an open neighborhood $U \subset \Omega$ such that for any $f \in \mathcal{D'}(U,\Lambda^{k}\R^{n}) $, the equation \eqref{aplic1} is continuously solvable in $U$ if and only if for every $\epsilon>0$ and every compact set $K \subset \subset U$, there exists $\theta=\theta(K,\epsilon)>0$ such that one has, for every $\varphi \in \calD_{K}(U,\Lambda^k\R^n)$:
\begin{equation}\nonumber
\left| f(\varphi) \right|  \leq \theta \|\varphi\|_{{1}}+\epsilon(\|d_{\bsL,k}\varphi\|_{1}+\| d^{*}_{\bsL,k-1}\varphi \|_{{1}}).
\end{equation}
\end{corollary}

Theorem \ref{teo6.1} is a direct consequence of the previous result, taking $k=0$ (recall that one has $d_{\bsL,0}=\nabla_{\bsL}$ and $d^{*}_{\bsL,0}=\diver_{\bsL^{*}}$). We emphasize that the operator is not canceling when $k=1$ or $k=n-1$ (see \cite[Section 4]{HP3}).

\bigskip

{\small
\parbox[t]{3.5in}
{\textbf{Laurent Moonens}\\
Universit\'e Paris-Saclay\\
Laboratoire de Math\'ematique UMR~8628\\
B\^atiment 307 (IMO)\\
rue Michel Magat\\
F-91405 Orsay Cedex (France)\\
E-mail: Laurent.Moonens@universite-paris-saclay.fr}
\bigskip

{\small
\parbox[t]{5in}
{\textbf{Tiago Picon}\\
University of S\~ao Paulo\\
Faculdade de Filosofia, Ci\^encias e Letras de Ribeir\~ao Preto\\ 
Departamento de Computa\c{c}\~ao e Matem\'atica\\
Avenida Bandeirantes 3900, CEP 14040-901,
Ribeir\~ao Preto, Brasil\\
E-mail: picon@ffclrp.usp.br}
\bigskip

\end{document}